\documentclass[a4paper,11pt]{amsart}
\usepackage[matrix,arrow,tips,curve]{xy}
\usepackage[english]{babel}
\usepackage[leqno]{amsmath}
\usepackage{amssymb,amsthm}
\usepackage{amscd}
\usepackage{enumerate}

\input{xy}
\xyoption{all}

\newtheorem{thm}[equation]{Theorem} 
\newtheorem{lemma}[equation]{Lemma} 
\newtheorem{prop}[equation]{Proposition} 
\newtheorem{cordef}[equation]{Corollary - Definition} 
\newtheorem{fact}[equation]{Fact} 
\theoremstyle{definition} 
\newtheorem{defi}[equation]{Definition} 
\newtheorem{remark}[equation]{Remark} 
\numberwithin{equation}{section} 
\newtheorem{example}[equation]{Example} 

\newcommand{\la}{\longrightarrow}

\newcommand{\ov}{\overline}

\newcommand{\Aut}{\operatorname{Aut}}
\newcommand{\Out}{\operatorname{Out}}
\newcommand{\Inn}{\operatorname{Inn}}
\newcommand{\codim}{\operatorname{codim}}

\newcommand{\pr}[1]{\mathbb{P}^{#1}}

\newcommand{\C}{\mathbb{C}}

\def\O{\mathcal O}

\def\X{\mathcal X}
\def\Y{\mathcal Y}
\newcommand{\Z}{\mathbb{Z}}
\newcommand{\R}{\mathbb{R}}

\newcommand{\G}{\Gamma }
\def\Is{{\mathcal I}so}

\newcommand{\sing}{X_{{\rm sing}}}

\def\mo{\underline{0}}

 \newcommand{\Mgbst}{\ov{\mathcal{M}_g}}
 
\def\Mgb{\overline{M}_g}
 
 \newcommand{\Mgnb}{\ov{M}_{g,n}}
  \newcommand{\Mgn}{{M_{g,n}}}

\newcommand{\Cgk}{C_g^{(k)}}

\newcommand{\Mt}{M^{\rm{{trop}}}}

\newcommand{\Mgpnn}{M_{g,n}^{\rm pure}}

\newcommand{\Mgtn}{{M_{g,n}^{\rm trop}}}
\newcommand{\Mgnt}{{M_{g,n}^{\rm trop}}}

\newcommand{\Mgrnn}{M_{g,n}^{\rm reg}}

\newcommand{\Mgtnn}{{M_{g,n}^{\rm trop}}}
\newcommand{\Mgtnb}{\ov{M_{g,n}^{\rm trop}}}

\newcommand{\wS}{w_{/S}}
\newcommand{\GS}{\Gamma_{/S}}

\begin{document}
\setcounter{page}{1}
%
%
\long\def\replace#1{#1}

%
%
\title{\replace{Algebraic  and tropical curves: comparing their moduli spaces}}
%
%
\author{\replace{Lucia Caporaso}}
\address{\replace{Dipartimento di Matematica - Universit\`a Roma Tre - Roma - ITALY}}
\email{\replace{caporaso@mat.uniroma3.it}}

%
%

\begin{abstract}
	
\replace{We construct  the moduli space for equivalence classes of  $n$-pointed tropical curves of genus 
$g$, together with its bordification given by  weighted  tropical curves, and its compactification by extended (weighted) tropical curves. We compare it to the moduli spaces of  smooth and stable,  $n$-pointed algebraic curves, from the combinatorial, the topological, and   the Teichm\"uller point of view.
}

\end{abstract}  

\maketitle
\thispagestyle{empty}
\tableofcontents

\section{Introduction}
\subsection{Overview of the paper}
This is a largely expository  article whose main goal is to construct the moduli space of tropical curves with marked points,
and to compare it to the moduli space of algebraic   curves, highlighting the symmetries and  the analogies
which  occur at various places.

The moduli space of nonsingular algebraic curves  is a well known
object that has been thoroughly studied, together with its compactifications, over the last five decades (to say the least).
There are several expository papers and monographs    attesting the depth and richness of the subject.
In this paper we shall give the basic definitions and state some famous results in a way that should be accessible to the non expert reader.
We will provide references   where details can be found. 

On the other hand tropical geometry is a rather young branch of mathematics, which has seen a flourish of diverse activities
over the recent years. Thus, quite naturally,    the field still lacks exhaustive bibliographical sources,
and even solid foundations. Hence we shall treat   tropical curves and their moduli quite thoroughly,
including technical details and proofs,
although some of the results we shall describe are known, if maybe only  as folklore.

Let us illustrate the content of the paper, section by section.
We begin 
with algebraic nodal  curves
of genus $g$ with $n$ marked points,   define the notion of stably equivalent curves, and introduce  the moduli spaces 
for these equivalence classes, the well known Deligne-Mumford spaces
$\Mgnb$.

Next, in Section~\ref{trpsec},
after some preliminaries about graphs,     we give the original definition of     abstract tropical curves
and    of tropically equivalent curves. Tropical equivalence should be viewed as the analogue of stable equivalence. Tropical curves up to tropical equivalence are parametrized by metric graphs with no vertex of valence less than $3$.
We extend the set-up to curves with marked points.

As it turns out, tropical curves do not behave well in families, in the sense that the genus may drop
under specialization.  One solution to this problem, introduced in \cite{BMV} for unpointed curves,
is to add  a weight function on the  vertices of their corresponding graph, thus generalizing the original notion  (when the weight function is zero we get back the original definition). 
We do that in the last part of Section~\ref{trpsec}.
Of course, this may not be the only, or the best, solution.
But it is one that can be worked out, and whose analogies with   moduli of stable algebraic curves are strong and interesting.
To distinguish, we refer to the original definition as ``pure" tropical curve,
and to the generalized one as ``weighted" tropical curve.

In Section~\ref{modsec} we explicitly
construct, as topological spaces,
 the moduli space for $n$-pointed weighted tropical curves of genus $g$, denoted by $\Mgtn$, 
and the analog  for pure tropical curves, denoted by $\Mgpnn$.
It is clear from the construction   that these two spaces are almost never topological manifolds; but they turn out  to be
  connected, Hausdorff, of pure dimension $3g-3+n$.
Also, $\Mgpnn$ is open and dense in $\Mgtn$; see Theorem~\ref{Mgt}.

Furthermore, $\Mgpnn$ and $\Mgtn$ are not tropical varieties in general.
The case $g=0$, which has been extensively studied
(see  for example  \cite{MIK3}, \cite{MIK5}, \cite{GM}, \cite{GKM}),
 is an exception, because of the absence of automorphisms.
Indeed  $\Mt_{0,n}$ (which coincides with
$M_{0,n}^{\rm{ pure}}$) is   a tropical variety for every $n\geq 3$.

In the literature there are a few constructions   giving these moduli spaces some special structure, resembling that of a tropical variety;
see \cite{GKM}, \cite{BMV} and \cite{chan} for example; nevertheless we
   prefer to treat them just as topological spaces, as the categorical picture does not look clear at the moment. 

Essentially by definition, $\Mgtn$ is closed under specialization, but it is not compact. A natural compactification  for it, by what we call ``extended tropical curves",   is studied in subsection~\ref{comp}. 

In Section~\ref{compsec} we compare the space  $\Mgtn$ to the Deligne-Mumford space  $\Mgnb$. Both spaces admit a natural partition
indexed by stable graphs,
with a poset structure defined by the inclusion of closures. The properties 
of the two partitions, and the correspondence 
  between them, 
 is described in Theorem~\ref{corr}.

Section~\ref{secteich} considers the Teichm\"uller approach to the moduli space of smooth curves $M_g$,
and to the moduli space of metric graphs.
While the Teichm\"uller point of view constitutes  one of the principal chapters of the   complex algebraic theory,
its analog for metric graphs, by means of the Culler-Vogtmann space (see \cite{CuVo}),
is less known in tropical geometry, and awaits to be further studied. 
We conclude the  paper with a table summarizing these  analogies, and with a list of open problems and research directions.

\noindent
{\it Acknowledgements.} I wish to thank M. Melo, G. Mikhalkin and F. Viviani   for several useful  comments,
and a referee for a very detailed report.

\subsection{Algebraic curves}
Details about the material of this and the next subsection may be found, for example, in   \cite[Chapt 1-4]{HM} or in  \cite[Chapt. 10 and 12]{gac}.

By an ``algebraic curve", or simply ``curve" when no confusion is possible, in this paper we mean a reduced, connected projective variety 
(not necessarily irreducible) of dimension one, defined over an algebraically closed field.
The arithmetic genus  of a curve $X$ is $g_X=h^1(X,\O_X)$.

We denote by $\sing$ the set of singular points of $X$.

We will only be interested in nodal curves, i.e. curves  having at most ordinary double points as singularities.
The reason is that the class of   nodal curves, up to an important equivalence relation which we will define in subsection~\ref{stab}, has a moduli space with very nice properties: it is a projective, irreducible (reduced and normal)
variety containing as a dense open subset the moduli space of nonsingular curves.
The need for an equivalence relation is quite fundamental, as the set of all nodal curves is too big to admit a separated moduli space. It is in fact easy to produce families of nodal, or even nonsingular, curves admitting (infinitely many) non isomorphic nodal curves as limits.

Before being more specific, we want to extend the present discussion to so-called $n$-pointed curves.

Let $g$ and $n$ be nonnegative integers.
A {\it nodal $n$-pointed curve} of genus $g$, denoted by $(X;p_1,\ldots,p_n)$, is a nodal curve $X$ of arithmetic genus $g$, together with $n$ distinct nonsingular points $p_i\in X$, $i=1,\ldots, n$.

The nodal $n$-pointed curve  $(X;p_1,\ldots,p_n)$ is {\it stable} (in the sense of Deligne and Mumford)
if the set of automorphisms of $X$ fixing $\{p_1,\ldots,p_n\}$ is finite;
or, equivalently, if the line bundle $\omega_X(p_1+\ldots + p_n)$ has positive degree on every subcurve of $X$
(where $\omega_X$ is the dualizing line bundle\footnote{For smooth curves $\omega_X$  coincides with the canonical bundle.} of $X$).

Finally, since for any irreducible component $E\subset X$ we have
\begin{equation}
\label{degdual}
\deg_E\omega_X(\sum_1^n p_i)=2g_E-2+|E\cap \big\{p_1,\ldots , p_n\}|+|E\cap\ov{X\smallsetminus E}|
\end{equation}
we have that our nodal $n$-pointed curve  is stable
 if and only if for every irreducible component $E$ of arithmetic genus at most $1$ we have

$$|E\cap \{p_1,\ldots, p_n\}|+|E\cap\ov{X\smallsetminus E}|\geq
 \begin{cases}
3 & \text{ if $g_E=0$}   \\
1 & \text{ if $g_E=1$.}
\end{cases}
$$
Note that the requirement for $g_E=1$ is saying that there exist no stable curves for $g=1$ and $n=0$.
More generally it is easy to check that stable curves exist if and only if $2g-2+n\geq 1$.
\begin{fact}
\label{mgnb}
Assume $2g-2+n\geq 1$.
There exists an irreducible projective scheme
of dimension $3g-3+n$, denoted by $\Mgnb$, which is the
(coarse) moduli space of $n$-pointed stable curves of genus $g$.
The moduli space of nonsingular $n$-pointed curves of genus $g$ is an open (dense) subset $\Mgn\subset \Mgnb$. 
\end{fact}
\begin{remark}
\label{moduli}

What do we mean by ``moduli space"?
In algebraic geometry there are various categories 
(schemes, stacks, algebraic spaces)
in   which moduli spaces can be axiomatically defined and their
properties rigorously proved. The basic axiom  in the schemes category  is, of course,
the fact that closed points\footnote{I.e. points defined over the base   field, versus points defined over some field extension of it.} be in bijection with the isomorphism classes of the objects under investigation
(stable pointed curves, in our case).
Other axioms require  that to a family of objects  parametrized by a scheme $B$
there is a unique associated morphism from $B$ to the moduli space,
and ensure that a moduli space is unique when it exists.

The    word ``coarse" in the statement 
indicates that $n$-pointed stable curves may have nontrivial automorphisms,
so that to a morphism from $B$ to $\Mgnb$ there may correspond more than one family of 
$n$-pointed stable curves,
or no  family at all.
We refer to \cite{HM} for    details about this topic.

By contrast, as of this writing, in tropical geometry there is no clear understanding of what a good categorical framework for moduli theory could be. Therefore we shall simply view tropical moduli spaces of curves as geometric objects (specifically: topological spaces) whose points are in bijection with isomorphism classes of pointed tropical curves up to tropical equivalence.
Essentially by construction,
the topological structure of our  tropical moduli spaces will reflect the most basic notion of a continuosly varying family of tropical curves.

The results stated in Fact~\ref{mgnb} may be considered a great achievement 
of twentieth century algebraic geometry.
They are due to P. Deligne - D. Mumford \cite{DM}, F. Knudsen \cite{Kn}, and  D. Gieseker \cite{Gie},
with fundamental contributions from other mathematicians (among whom A. Grothendieck and  A. Mayer). 
 Some details about the construction of the moduli space of stable curves will be given at the beginning
of Section~\ref{secteich}.
\end{remark}
\begin{remark}
It is important to keep in mind that the   points $p_1,\ldots,p_n$ of  a stable $n$-pointed curves are always meant to be labeled
(or ordered), as the following example illustrates.
\end{remark}
\begin{example}
\label{04}
Let $g=0$. Then there exist stable curves only if $n\geq 3$,
and if $n=3$ there exists a unique one (up to isomorphisms): $\pr{1}$ with $3$ distinct points,
indeed it is well known that the automorphisms of $\pr{1}$  act transitively on the set of triples of distinct points.
Notice also that   there exist  no nontrivial automorphisms of $\pr{1}$ fixing $3$ points,
therefore 
{\it for every $n\geq 3$ a smooth $n$-pointed curve of genus $0$ has no nontrivial automorphisms.}

The next case, $n=4$, is quite interesting.
We may set the three first points to be $p_1=0,p_2=1,p_3=\infty$ and let the fourth point $p_4$ vary in 
$\pr{1}\smallsetminus \{0,1,\infty\}$. So we have a simple description of $M_{0,4}$ as a nonsingular, rational,
affine curve. Its completion, $\ov{M}_{0,4}$, the moduli space of stable curves, is isomorphic to $\pr{1}$, and the three boundary points are
the three singular $4$-pointed stable rational curves such that $X=C_1\cup C_2$ with $|C_1\cap C_2|=1$
and, of course, $C_i\cong \pr{1}$. What distinguishes the three curves from one another is the distribution of the four points
among the two components. Since each of them must contain two of the four points, and since the two components can be interchanged, we have
  three different non-isomorphic curves corresponding to the three different partitions
$(p_1,p_2|p_3,p_4),\  (p_1,p_3|p_2,p_4),\  (p_1,p_4|p_2,p_3)$.
\end{example}
\begin{example}
If $g=1$ there exist stable curves only if $n\geq 1$, and for $n=1$ we have a $1$-dimensional family of smooth ones, and a unique singular one: an irreducible curve of arithmetic genus $1$ with one node and one marked point.
\end{example}
\subsection{Stably equivalent nodal curves.}
\label{stab} 
We now describe an explicit procedure to construct   the  {\it stabilization}, or {\it stable model}
of a nodal pointed curve. As we mentioned, this is needed to have a separated moduli space
for equivalence classes of nodal curves.
We will do this by a sequence of steps, which  will be useful later on.

Let $(X;p_1,\ldots,p_n)$ be an $n$-pointed nodal curve of genus $g$, with $2g-2+n\geq 1$. Suppose $X$ is not stable.
Then  $X$  has  components  where the degree of  $\omega_X(\sum_1^n p_i)$ is at most $0$.

$\bullet$ {\it Step 1.}
Suppose there exists a component $E$ of $X$ where this degree is negative. Then, by (\ref{degdual}),  $E$
is a smooth rational component  containing none of the $p_i$ and such that
  $|E\cap\sing|=1$ (so, $\deg _E\omega_X(\sum_1^n p_i)=-1$).  These components are called
 unpointed  rational tails.
 We remove $E$ from $X$; this operation does not alter the genus of $X$,  its connectedness 
 (as   $|E\cap \overline{X\smallsetminus E}|=1$),
  the nature of its singularities, or the  number of marked points
  (which remain nonsingular and distinct). On the other hand it may create one new unpointed rational tail,
  in which case we repeat the  operation.
  So, the first step consists in iterating
  this operation until there are no   unpointed  rational tails remaining. It is clear that this process terminates after finitely many steps.
We denote by
  $(X';p_1',\ldots,p_n')$ the resulting $n$-pointed curve, which is unique up to isomorphism.   
  
 $\bullet$   {\it Step 2.}
  Now the degree of $\omega_{X'}(\sum_1^n p_i')$ on every component of $X'$ is at least $0$,
and if $(X';p_1',\ldots,p_n')$ is not stable there exists a  component $E$ on which this degree is $0$.
By (\ref{degdual}) we have $E\cong \pr{1}$ and $E$ contains a total of  $2$ among marked and singular points of $X'$.
There are two cases according to whether $E$ contains a marked point or not
($E$ necessarily contains at least one singular point of $X'$, for $E\subsetneq X'$).

$\bullet$ $\bullet$ {\it Step 2.a.}
Suppose $p'_i\not\in E$ for all $i=1,\ldots, n$. Then   $E$ intersects $\overline{X'\smallsetminus E}$ in $2$ points
(for $X'$ contains no unpointed  rational tail).
Such an $E$ is are called
 exceptional component.
This step consists in contracting each of these exceptional components to a node. Notice that this step does not touch the marked points, nor does it add any
new rational tail.
The resulting curve $(X'';p_1'',\ldots,p_n'')$ is again an $n$-pointed genus $g$ curve free from 
 exceptional components or unpointed rational tails.

$\bullet$ $\bullet$ {\it Step 2.b.}
Suppose  $E$ contains some of the marked points. Since $E$ must intersect $ \overline{X''\smallsetminus E}$,
we have   $|E\cap \overline{X''\smallsetminus E}|=1$ and $E$ contains only one marked point, $p''_1$ say.
These components are called
 uni-pointed  rational tails.
Now we  remove $E$ as in the first step,
but we need to keep track of $p_1''$. To do that we mark
the attaching point $E\cap \overline{X''\smallsetminus E}$, which, after we remove $E$, becomes   
a nonsingular point    replacing   the lost point  $p_1''$.
By iterating this process until there are no  uni-pointed rational tails left,
we arrive at a genus $g$ stable curve $(X''';p_1''',\ldots,p_n''')$.
It is easy to check that this curve  is unique up to isomorphism.

This curve  $(X''';p_1''',\ldots,p_n''')$ is called the stabilization, or stable model, of the given one.

\begin{remark}
\label{stabeq}
Following \cite{HM}, we say that two $n$-pointed  nodal curves are {\it stably equivalent} if they have the same stabilization (always assuming $2g-2+n\geq 1$). 
It is easy to check that two stable curves are stably equivalent only if they are isomorphic.
So, in every stable equivalence class of $n$-pointed curves there is a unique stable representative.

By Fact~\ref{mgnb}, 
there exists a projective   variety, $\Mgnb$, parametrizing stable equivalence classes of $n$-pointed
nodal curves of genus $g$.
\end{remark}
\begin{remark}
\label{ind}
There are many good reasons for  
extending our field of interest from curves to pointed curves. Here is a basic  and useful   one:

{\it Stability of  pointed curves is preserved under normalization}.

More precisely,
let  $(X;p_1,\ldots,p_n)$ be an $n$-pointed stable curve of genus $g$.
Pick a node $q\in \sing$, let $\nu_q: X^{\nu}_q\to X$  be the  normalization at $q$; let $q_1,q_2\in X^{\nu}_q$
be the two branches of $q$, and abuse notation by setting $p_i=\nu_q^{-1}(p_i)$ for $i=1,\ldots,n$.
Then one easily checks that the $(n+2)$-pointed curve $(X^{\nu}_q;p_1,\ldots,p_n,q_1,q_2)$ is either   stable of genus $g-1$
(if $q$ is not a separating node of $X$), or the disjoint union of two pointed stable curves
of genera summing to $g$
(if $q$ is a separating node). 
\end{remark}

  \section{Graphs and pure tropical curves }
  \label{trpsec}
\subsection{Graphs}
\label{contr}
By a {\it topological graph}, or simply a graph, we mean  a one dimensional finite simplicial (or CW) complex $\Gamma$;
we denote by $V(\Gamma)$
 the set of its vertices (or $0$-cells)  and by   $E(\Gamma)$ the set of its edges (or $1$-cells). 
 To every $e\in E(\Gamma)$ one associates the  
  pair $\{v,v'\}$ of possibly equal vertices which form the boundary of $e$; we call $v$ and $v'$ the {\it endpoints} of $e$.
 If $v=v'$ we say that $e$ is a {\it loop-edge} based at $v$.

 The {\it valence} of a vertex $v$ is the number of edges having $v$ as endpoint, with the convention that a loop 
  based at $v$ be counted  twice.
 
 For a fixed integer $p$, 
 we say that a vertex of valence $p$ is {\it $p$-valent}, and 
  we say that a graph  is 
  {\it $p$-regular} if all of its vertices are $p$-valent.
  
 The genus  $g(\Gamma)$ of $\Gamma $ is   its first Betti number
$$
g(\Gamma)=b_1(\G):={ \operatorname{rk}} _{\Z}H_1(\G, \Z) =|E(\Gamma)| -|V(\Gamma)|  +c
$$
where $c$ is the number of connected components of $\Gamma$. 
Morphisms between topological graphs are, as usual,  cellular maps.

We also need to define graphs in purely combinatorial terms, for which there are various possibilities.
Our definition,   almost   the same as in \cite{gac}, is most convenient to
simultaneously treat tropical and algebraic curves. 
\begin{defi}
A  {\it combinatorial graph}  $\Gamma$ with $n$ {\it legs} is the following set of data:
\begin{enumerate}
\item
A finite non-empty set $V (\Gamma)$, the set of {\it vertices}.
  \item
   A finite set $H(\Gamma)$, the set of {\it half-edges}.
  \item 
An involution 
   $$
\iota: H(\Gamma)\la H(\Gamma)\  \  \  \  \  h\mapsto \overline{h} 
 $$ 
 with  $n$ fixed points, called   {\it legs},  whose set  is denoted by $L(\Gamma)$.
   \item
An {\it endpoint} map $\epsilon:H(\Gamma)\to V(\Gamma)$.
 \end{enumerate}

 A pair $e=\{h ,  \overline{h}\}$ of distinct elements   in $H(\Gamma)$ interchanged by  the involution is called
an {\it edge} of the graph; the set of edges is denoted by $E(\Gamma)$.
 If $\epsilon(h)=v$ we say that $h$, or $e$, is {\it adjacent} to $v$.

The {\it valence} of a vertex   $v$ is the number $|\epsilon^{-1}(v)|$ of half-edges adjacent to $v$.

An edge adjacent to a vertex of valence $1$ is called a {\it leaf}.

An edge whose endpoints coincide is called a {\it loop-edge}.

Two legs are called {\it disjoint} if their endpoints are distinct.
 \end{defi}
 \begin{defi}
 \label{combg}
 A morphism $\alpha$ between combinatorial graphs  $\Gamma$ and $\Gamma'$
 is a map $\alpha:V(\Gamma)\cup H(\Gamma)\to V(\Gamma')\cup H(\Gamma')$
 such that $\alpha(L(\Gamma))\subset L(\Gamma')$, and such that
the two diagrams below are commutative.
 \begin{equation}\label{diag1}
\xymatrix{
V(\Gamma)\cup H(\Gamma) \ar@{^{}->} ^{\alpha}[r]  \ar@{->}^{(id_V,\epsilon)}[d] & V(\Gamma')\cup H(\Gamma')\ar@{^{}->}^{(id_{V'},\epsilon')}[d] \\
V(\Gamma)\cup H(\Gamma) \ar@{^{}->}^{\alpha}[r] & V(\Gamma')\cup H(\Gamma')
}
\end{equation}
In particular,  $\alpha(V(\Gamma))\subset V(\Gamma')$. Next
\begin{equation}\label{diag2}
\xymatrix{
V(\Gamma)\cup H(\Gamma) \ar@{^{}->} ^{\alpha}[r]  \ar@{->}^{(id_V,\iota)}[d] & V(\Gamma')\cup H(\Gamma')\ar@{^{}->}^{(id_{V'},\iota')}[d] \\
V(\Gamma)\cup H(\Gamma) \ar@{^{}->}^{\alpha}[r] & V(\Gamma')\cup H(\Gamma')
}
\end{equation}
We say that a morphism $\alpha$ as above is an {\it isomorphism} if $\alpha$ induces, by restriction,
three bijections   
$\alpha_V:V(\Gamma)\to V(\Gamma')$, $\alpha_E:E(\Gamma)\to E(\Gamma')$ and $\alpha_L:L(\Gamma)\to L(\Gamma')$.

An automorphism of $\G$ is an isomorphism of $\G$ with itself.
 \end{defi}
\begin{remark}
The image of
an edge $e\in E(\Gamma)$ is either an edge, or a vertex $v'$ of $\Gamma'$;
in the latter case the endpoints of $e$ 
are also also mapped to $v'$, and we say that $e$ is contracted by $\alpha$.
\end{remark}
A  trivial  example of morphism is the map  forgetting some legs.
\begin{example}
\label{invertloop}
Let $e=\{h ,  \overline{h}\}$ be a loop-edge of $\G$. Then $\G$ has a ``loop-inversion" automorphism,
  exchanging $h$ and $ \overline{h}$ and fixing everything else.
\end{example}

It is clear that to every topological    graph we can associate a unique combinatorial graph
with no legs.
 
Conversely, to every combinatorial graph   with vertex set $V$ and edge set $E$,  we can associate  a unique topological graph, as follows.
We take $V$ as the set of  $0$-cells; then we add a $1$-cell for every $e=\{ h, \overline{h} \}\in E$,
in such a way that the boundary of this $1$-cell is $\{ \epsilon(h), \epsilon(\overline{h})\}$.

Now, if the combinatorial graph has a non empty set of legs $L$, we add to the topological graph associated to it 
a $1$-cell for every $h\in L$ in such a way that one extreme of the $1$-cell contains $\epsilon(h)$ in its closure.
The topological space we obtain in this way will be called a  topological  graph with $n$ legs.
So, legs are $1$-cells only one end of which is adjacent to a vertex (which will be called the endpoint of the leg).
 From now on we shall freely switch between the combinatorial and topological structure on graphs, intermixing the two points of view without mention.
The notions of valence, genus, endpoints, $p$-regularity and so on, can be given in each setting and coincide.

\begin{remark}
From now on, we shall  assume our graphs to be connected, unless we specify otherwise.
\end{remark}

We now describe a type of morphism between  graphs,   named  {\it (edge)  contraction},   which will play an important role. We do that for topological graphs, leaving the translation for combinatorial graphs to    the reader.

Let $\Gamma$ be a topological graph and
   $e\in E(\Gamma)$ be an edge. Let 
$\Gamma_{/e}$ be the graph obtained by contracting $e$ to a point and leaving everything else unchanged  \cite[sect I.1.7]{Die}. Then there is a natural continuous surjective map
  $\Gamma \to \Gamma _{/e}$, called    the {\it contraction} of $e$.
  More generally, if $S\subset E(\Gamma)$ is a set of edges, we denote by $\GS$ the contraction of every edge in $S$ and denote by  $\sigma:\Gamma\to \GS$ the associated map.
  Let $T:=E(\Gamma)\smallsetminus S$. Then there is a natural identification between
  $E(\GS)$ and $T$. Moreover the map $\sigma$ induces a bijection between the sets of legs $L(\Gamma)$ and $L(\GS)$, and a surjection
  $$
  \sigma_V:V(\Gamma)\la V(\GS);\  \  \  v\mapsto \sigma(v).
  $$
  Let us disregard the legs, as they play essentially no role in what we are going to describe.
Notice that every connected component of 
  $\Gamma - T$ 
  (the graph obtained from $\Gamma$ by  removing every edge in $T$)
  gets contracted to a vertex of $\GS$; conversely, for every vertex $\ov{v} $ of $\GS$ its preimage $\sigma^{{-1}}(\ov{v})\subset \Gamma$ is a connected component of  $\Gamma - T$.
  In particular, we have
  \begin{equation}
  \label{gendec}
b_1(\Gamma- T)=\sum_{\ov{v}\in V(\GS)}b_1(\sigma^{{-1}}(\ov{v})).
\end{equation}

Let $\sigma:\Gamma\to \GS$ be the contraction of $S$ as above.
Then
\begin{equation}
 \label{lme} b_1(\Gamma)=b_1(\GS)+b_1(\Gamma- T). 
 \end{equation}

\subsection{Tropical curves}
 \label{tropeq}
 
A {\it metric  graph}  is a leg-free graph $\Gamma$
 endowed with the structure of a metric space, so that every edge is locally isometric to an interval in $\R$, and where the distance between two points is the  shortest length of an edge-path joining them.
In particular, on a metric graph we have a length function 
$$
\ell:E(\Gamma)\to \R_{> 0}
$$
mapping an edge to the distance between its endpoints. It is clear, conversely, that the datum of such an $\ell$
determines on $\Gamma$ the structure of a metric space with every edge $e$ having length $\ell(e)$.

If $\Gamma$ has legs, it is convenient to extend the function $\ell$ to the legs of $\Gamma$, by setting
$\ell(x)=\infty$ for every $x\in L(\Gamma)$.

The genus of a metric graph is the genus of the underlying topological graph.

\

 We shall now define tropical curves   following G. Mikhalkin.
 An abstract (pure\footnote{A explained in the introduction, ``pure" is added for reasons that will be clear later.}) tropical curve   is   almost the same as a metric graph; see 
 \cite[Prop. 5.1]{MIK3} or \cite[Prop. 3.6]{MZ}.
 The difference is in the length of those edges adjacent to vertices of valence $1$,
i.e. the  leaves; the length of a leaf  is set to be equal to $\infty$ for a tropical curve, whereas for a metric graph is
 finite.

 \begin{defi}
 \label{tropc}
A {\it (pure) tropical curve} of genus $g$ is a pair $(\Gamma, \ell)$ where $\Gamma$ is a 
leg-free graph
of genus $g$ 
and $\ell$ a length function  on the edges $\ell:E(\Gamma)\to \R_{> 0}\cup\{ \infty\}$
 such that $\ell(e)=\infty$ if and only if $e$ is a leaf.  

Two tropical curves are {\it (tropically) equivalent} if they can be obtained  from one another by adding or removing  vertices of valence $2$, or   vertices of valence $1$ together with their adjacent leaf.
\end{defi}

 \begin{remark}
 \label{metsp}
A tropical curve is not a metric space, as  the distance from a $1$-valent vertex to another point   is not defined. 
Of course, if  all $1$-valent vertices are removed what remains is a metric space.
We shall see in \ref{rep3} that equivalence classes of tropical curves are bijectively parametrized by certain metric graphs.
 \end{remark}
   Let us illustrate tropical equivalence in details with some pictures.

\begin{enumerate}[{\bf (1)}]
\item
{\it Addition/removal  of a vertex of valence $1$ and of of its adjacent edge (a leaf).}
The next picture illustrates the removal of the one-valent vertex $u_0$ 
and of the leaf $e_0$ adjacent to it.
\begin{figure}[!htp]
\label{}
$$\xymatrix@=1pc{
&&& *{\bullet}\ar@{.}[d]\ar@{-}[rr]^>{u}^<{v} ^<(0.5){e_1} &&*{\bullet}\ar@{-}[d]_{e_0}^>(1.1){u_0}    \ar@{-}[rr]^>{w}^<(0.4){e_2} &&*{\bullet}\ar@{.}[d]&&
\stackrel{}{\la} &
& &&*{\bullet}\ar@{.}[d]\ar@{-}[rr]^>{u}^<{v} ^<(0.5){e_1} &&*{\bullet}    \ar@{-}[rr]^>{w}^<(0.4){e_2} &&
*{\bullet}\ar@{.}[d]\\
&&&&&*{\bullet}&&&&&&&&&&&&&&&&&&&&&&&&&&&&&&\\
}$$
\caption{Removal  of the   $1$-valent vertex $u_0$ and of its adjacent edge $e_0$.}
\end{figure}
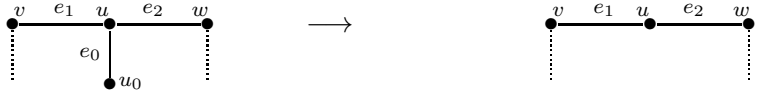
The opposite move is the addition of a leaf, where the length of the added edge $e_0$ is  set equal to $\infty$.
\item
{\it Addition/removal of a vertex of valence $2$.}
Pick an edge  $e\in E(\Gamma)$  and denote by $v,w\in V(\Gamma)$ its endpoints. We can add a vertex $u$ in the interior of $e$.
This move replaces the edge $e$ of length $\ell(e)$ by two edges $e_v$ 
(with endpoints $v,u$) and $e_w$ (with endpoints $w,u$),
whose lengths satisfy $\ell(e)=\ell(e_v)+\ell(e_w)$.
If one of the two endpoints of $e$, say $v$,  has valence $1$ we set the length of $e_v$ equal to $\infty$, whereas the length of $e_w$ can be arbitrary.
The opposite procedure, which should    be clear, is represented in the figure~\ref{add2} below:
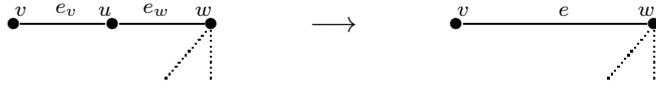
\begin{figure}[!htp]
\label{add2}
$$\xymatrix@=1pc{
&&&*{\bullet}\ar@{-}[rr]^>{u}^<{v} ^<(0.5){e_v} &&*{\bullet}   \ar@{-}[rr]^>{w}^<(0.4){e_w} &&*{\bullet}\ar@{.}[d]\ar@{.}[dl] &&
\stackrel{}{\la} && 
*{\bullet}\ar@{-}[rrrr]^>{w}^<{v}&&&&*{\bullet}\ar@{.}[dl]\ar@{.}[d]  \ar@{-}[l]_>(2.2){e}\\
&&&&&&&&&&&&&&&&&&&&&&&&
}$$
\caption{ Removal of a vertex of valence $2$.}
\end{figure}
\end{enumerate}

\subsection{Pointed tropical curves}
Points on a tropical curve are conveniently represented as legs on the corresponding graph.
Indeed, let $C$ be a tropical curve and $p\in C$; if $p$ is a vertex we add a leg based at $p$, if $p$ is in the interior of an edge, we add a  vertex at $p$ and a leg based at it; since the added vertex is $2$-valent, this operation does not change the equivalence class of $C$.   
 \begin{defi}
 \label{troppc}
An {\it $n$-pointed (pure) tropical curve} of genus $g$ is a pair $(\Gamma, \ell)$ where $\Gamma$ is a combinatorial graph
of genus $g$ with a set $L(\Gamma)=\{x_1,\ldots, x_n\}$ of  legs, 
and $\ell$ a length function
$$
\ell:E(\Gamma)\cup L(\Gamma)\la \R_{> 0}\cup\infty
$$
such that $\ell(x)=\infty$ if and only if $x$ is a leaf or a leg.

The legs $\{x_1,\ldots, x_n\}$ are called {\it (marked) points} of the curve.

We say that two pointed tropical curves   are {\it (tropically) equivalent}  if they can be obtained one from the other by adding or removing  vertices of valence $2$ or  vertices of valence $1$  with their adjacent leaves.

We say that two tropical curves  $(\Gamma, \ell)$
and $(\Gamma', \ell')$  
with $n$ marked points  $L(\Gamma)=\{x_1,\ldots, x_n\}$ and $L(\Gamma')=\{x_1',\ldots, x_n'\}$
are {\it isomorphic} if there exists an isomorphism $\alpha$ from $\G$ to $\G'$
as defined in \ref{combg},  such that $\ell(e)=\ell'(\alpha(e))$ for every $e\in E(\G)$, and
$\alpha(x_i)=x'_i$ for every $i=1,\ldots, n$  .
\end{defi}
\begin{remark}
Just as for  stable curves, marked points are always  labeled.
\end{remark}
Two  pointed tropical curves are equivalent
if they can  be obtained  from one another by a finite sequence of the two moves described in the pictures of subsection~\ref{tropeq}.
Notice that now in move (2) we allow adding  $2$-valent vertices in the interior of a leg, as well as removing the endpoint $v$ of a leg if $v$ has valence $2$,
but we do not allow adding or removing legs.

\begin{remark}
\label{addl}
Tropical equivalence 
preserves the number of marked points, but not  their being disjoint, as the next picture shows.
\end{remark}
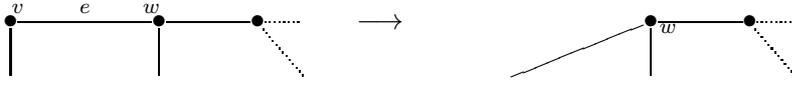
\begin{figure}[!htp]
$$\xymatrix@=1pc{
&*{\bullet}\ar@{-}[d]\ar@{-}[rrr]^>{w}^<{v}^(0,5){e}&&&*{\bullet}\ar@{-}[d]  \ar@{-}[rr]&&*{\bullet}\ar@{.}[dr]\ar@{.}[r]&&
\stackrel{}{\la} &
& &&  &*{\bullet}\ar@{-}[d]^<{w} \ar@{-}[dlll]\ar@{-}[rr]&&*{\bullet}\ar@{.}[dr]\ar@{.}[r]&&\\
&&&&&&&&&&&&&&&&&&&&&&&&
}$$
\caption{After removing  the $2$-valent vertex $v$  the two legs are no longer disjoint}
\end{figure}

\begin{example}
\label{0triv}
Let $\Gamma$   have one vertex and no edges or legs. Then $\Gamma$ is an unpointed tropical curve
of genus $0$,
equivalent to any tropical curve of genus $0$.
From the moduli point of view, its equivalence class is viewed as a trivial one, and will be excluded
in future considerations. 

Let now $\Gamma$ be a graph with one vertex, $v$, and one loop attached to it;
so $\Gamma$ is an unpointed tropical curve
of genus $1$. Now $v$ is $2$-valent, hence can be removed, leaving us with something which is not
a tropical curve. For this reason, these curves are viewed as degenerate, and will also be excluded.
Hence if $g=1$ we shall always assume $n\geq 1$.

By a similar reasoning, if $g=0$ we shall always assume $n\geq 3$. This motivates the future assumption
$2g-2+n\geq 1$.
\end{example}
\begin{prop}
\label{rep3} Assume $2g-2+n\geq 1$.
\begin{enumerate}
\item
\label{rep3a}
Every equivalence class of $n$-pointed pure tropical curve contains a representative whose $n$ marked points are disjoint.
\item
\label{rep3b}
The set of equivalence classes of $n$-pointed pure tropical curves of genus $g$ 
is in bijection
with the set of metric graphs of genus $g$ with $n$ legs having no vertex of valence less than $3$.
\end{enumerate}
\end{prop}
\begin{proof}
Let us prove (\ref{rep3a}). Pick a representative  $(\Gamma,\ell)$ such that
the legs of $\Gamma$ are not disjoint.
Hence   there is a vertex $v\in V(\Gamma)$ which is the endpoint of $m\geq 2$ legs,
$x_1,\ldots, x_m$. We add a vertex $v'_i$ (of valence $2$) in the interior of $x_i$ for all $i=2,\ldots, m$;
this operation does not    change the equivalence class. 
This creates, for $i=1,\ldots, m$, a new edge $e'_i$  and a leg $x_i'$ adjacent to $v'_i$. In this way we obtain a new graph $\Gamma'$ whose legs $x_1,x'_2,\ldots, x'_m$ are disjoint. Of course $\Gamma$ and $\Gamma'$ are tropically equivalent.
After repeating this process finitely many times we arrive at a graph with disjoint legs, tropically equivalent to the original one.

Now part (\ref{rep3b}).
Let $(\Gamma,\ell)$ be a graph with $n$ 
legs, i.e.
an $n$-pointed tropical curve.
 We shall describe a construction to produce, in the same equivalence class of the given curve,
a metric graph  with $n$ legs having no vertex of valence less than $3$.
This construction is entirely similar to the construction described in subsection~\ref{stab} for algebraic curves.
We itemize it so as to highlight the analogies.

$*$ {\it Step 1.}
Suppose $\Gamma$ has some $1$-valent vertex. Then we remove it, together with its leaf;
by definition this operation does not change the equivalence class.
We thus get a new  graph $\Gamma^*$, with a natural  inclusion  $E(\Gamma^*)\subset E(\Gamma)$ and a natural identification $L(\Gamma^*)= L(\Gamma)$.
We define the length function $\ell^*$ on $E(\Gamma^*)\cup L(\Gamma^*)$ by restricting $\ell$.
Hence $(\Gamma^*,\ell^*)$ is an $n$-pointed curve equivalent to the given one.
We can obviously iterate the above procedure until there are no $1$-valent vertices  left.
We denote by  $(\Gamma',\ell')$ the $n$-pointed tropical curve obtained at the end.

$*$ {\it Step 2.}
If $\Gamma'$ has no vertex of valence $2$ we are done.
So, let $v$ be a vertex of valence $2$ of $\Gamma'$.

$**$ {\it Step 2.a.}
If $v$ has no leg based at it, then we remove $v$ as described in subsection~\ref{tropeq}.
(i.e. in such a way that the two edges  adjacent to $v$ are merged into an  edge of length equal to the sum of their lengths). It is clear that this operation does not add any new leaf
or leg, and diminishes the number of these $2$-valent vertices. So we can repeat it finitely many times until
there are no  such  $2$-valent vertices left. The resulting tropical curve, $(\Gamma'',\ell'')$,
has $n$ 
points and it is tropically equivalent to the given one.

$**$ {\it Step 2b.}
If $v$ has a (necessarily unique) leg $l$ based at it, let $e$ be the (also unique) edge adjacent to $v$.
Now by removing $v$ the edge  $e$ changes into a leg, merging   with $l$ (see the figure in Remark~\ref{addl}).  
This operation preserves the equivalence class, the number of legs, and does not add any new leaf; it clearly diminishes
the number of $2$-valent vertices. So after finitely many iterations we arrive at a tropical curve, $(\Gamma''',\ell''')$ with $n$ 
points and such that the graph has no vertex of valence less than $3$.

We have thus shown that every tropical equivalence class of $n$-pointed tropical curves  has a representative
all of whose vertices have valence at least $3$. 
The uniqueness of  this representative is    trivial
to prove.
\end{proof}

\begin{remark}
\label{staban}
Recall from Section~\ref{stab} that two nodal algebraic curves are  stably equivalent  if they have isomorphic stabilizations.

{\it To say that two $n$-pointed tropical curves are tropically equivalent is analogous to say that two $n$-pointed algebraic curves are stably equivalent.}

This should be clear by comparing the construction of \ref{stab} with the proof of   Proposition~\ref{rep3}.
Indeed: (Step 1) the   removal  of a $1$-valent vertex and its adjacent edge corresponds to  
 removing an unpointed rational tail.  Next
(Step 2.a) the removal of a $2$-valent vertex 
adjacent to no legs
corresponds to contracting  an exceptional component to a node. 
 Finally (Step 2.b) the removal of a $2$-valent vertex   having an adjacent leg  corresponds to
 removing a uni-pointed rational tail.
    \end{remark}

\begin{example}
\label{04g}
Let $g=0$ and $n=4$.
A   graph  with $4$ legs and no vertex of valence $\leq 2$ can have at most $1$ edge.
The graph with $0$ edges is unique.
On the other hand there are three
non-isomorphic combinatorial graphs with one edge,
according to how the $4$ legs are distributed.
They are  drawn in the following picture.
Each of these graphs supports a one dimensional family of metric graphs, as the length of their unique edge 
varies in $\R_{>0}$.
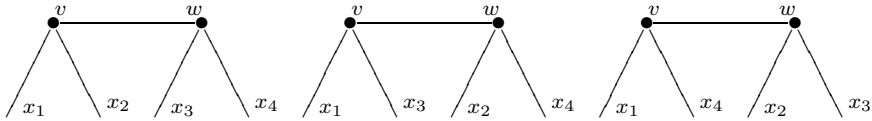
\begin{figure}[!htp]
$$\xymatrix@=1pc{
& *{\bullet}\ar@{-}[ddl]^(0,8){x_1}\ar@{-}[ddr]^(0,9){x_2}\ar@{-}[rrr]^>{w}^<{v}&&&*{\bullet}\ar@{-}[ddl]^(0,8){x_3}\ar@{-}[ddr]^(0,9){x_4}&&&
 *{\bullet}\ar@{-}[ddl]^(0,8){x_1}\ar@{-}[ddr]^(0,9){x_3}\ar@{-}[rrr]^>{w}^<{v}&&&*{\bullet}\ar@{-}[ddl]^(0,8){x_2}\ar@{-}[ddr]^(0,9){x_4}&&&
  *{\bullet}\ar@{-}[ddl]^(0,8){x_1}\ar@{-}[ddr]^(0,9){x_4}\ar@{-}[rrr]^>{w}^<{v}&&&*{\bullet}\ar@{-}[ddl]^(0,8){x_2}\ar@{-}[ddr]^(0,9){x_3}
\\
&&&&&&&&&&&&&&&&&&&&&&&&
\\
&&&&&&&&&&&&&&&&&&&&&&&&
}$$
\caption{The three genus $0$, $3$-regular graphs with  $4$ legs.}
\end{figure}
\end{example}
\begin{example}
{\it Rational tropical curves.}
A detailed description of the case $g=0$ and $n\geq 4$, with nice pictures, can be found  in \cite{MIK5}.

This case is special    for two reasons.
First: it is easy to see that an $n$-pointed rational pointed curve free from vertices of valence $\leq 2$
has no  nontrivial automorphisms. This is the key reason why
the  moduli space for these curves is a tropical variety.
As we saw  Example~\ref{04}, this   happens also for algebraic curves   and, as a consequence, the moduli spaces 
$\ov{M}_{0,n}$ are smooth varieties (see \cite{Kn}), whereas $\Mgnb$  is singular in general.
As we said in the introduction, this discrepancy seems to have its tropical  analogue in the fact that
moduli spaces of tropical curves of higher genus are not  tropical varieties in general.

The second reason why the genus zero  case is easier to handle is that  a family of genus zero  tropical  curves specializes to a genus zero curve. 
This is not the case if $g\geq 1$ as we are going to explain.
\end{example}
\subsection{Adding a weight function on tropical curves}
\label{tropsec}

The definition of pure tropical curve presents a problem when   studying families.
To explain why, let us first identify tropical curves with metric graphs having no vertex of valence less than $3$ (which, up to tropical equivalence, we can do).
Now, by varying the lengths of the edges of a metric graph we obtain
a   family.
  Let us make this precise;
 fix   a graph $ \Gamma$ as above, write $E(\Gamma)=\{e_1,\ldots, e_{|E(\Gamma)|}\}$,
 and   consider the space 
of all
metric graphs,  i.e. of all tropical curves, supported on it.
This space is easily identified with $\R^{E(\Gamma)}_{> 0}$, 
indeed,     to a vector 
$$(l_1,\ldots, l_{|E(\Gamma)|})\in \R^{E(\Gamma)}_{> 0}$$
there  corresponds  the tropical curve $(\Gamma,\ell)$ such that   $\ell(e_i)=l_i$, $\forall i$.
It is then natural to ask what happens when some of the lengths go to zero.
So, with the above notation, 
 let  $l_1$ say, tend to $0$. 
 How do we give the limit an interpretation  in tropical, or geometric, language?
There is  a simple candidate:
as $l_1$ tends to zero,   $(\Gamma,\ell)$ specializes to a metric graph 
$(\ov{\Gamma},\ov{\ell})$ where $\ov{\Gamma}=\Gamma_{/e_{1}}$ is   obtained    by contracting $e_1$ to a point,  and $\ov{\ell}(e_i)=\ell(e_i)$, \  $\forall i\geq 2$
(as 
$E(\ov{\Gamma})=E(\Gamma)\smallsetminus \{e_1\}=\{e_2,\ldots, e_{|E(\Gamma)|}\}$).
But there is a drawback with this limit: its genus may be smaller than $g(\Gamma)$.
Indeed we have 
$$g(\ov{\Gamma})=\begin{cases}
g(\Gamma)-1 & \text{ if $e_1$ is   a loop}   \\
g(\Gamma) & \text{ otherwise.}
\end{cases}
$$
From a geometric perspective this is quite unpleasant.
We like the genus to remain constant under specialization.
A solution to this problem is provided by S. Brannetti, M. Melo and F. Viviani in \cite{BMV}. The idea  is to extend the definition of a tropical curve by adding    a weight function on the   vertices. 

\begin{defi}
A {\it weighted graph} 
with $n$ legs is a pair  $(\Gamma, w)$
where $\Gamma$ is a graph 
with $n$ legs, and 
$w:V(\Gamma)\to \Z_{\geq 0}$ a {\it weight} function
on the vertices.

The genus $g(\Gamma, w)$    is    defined as follows:
\begin{equation}
\label{gw}
g(\Gamma, w)=b_1(\Gamma)+\sum_{v\in V(\Gamma)}w(v)=b_1(\Gamma)+|w|.
\end{equation}
A {\it weighted metric graph} is defined exactly in the same way, assuming that $\Gamma$ is a metric graph to start with.
  \end{defi}

 \begin{defi}
\label{cont}
Let $S\subset E(\Gamma)$ be a set of edges of a weighted graph $(\Gamma,w)$.
Using the notations of subsection~\ref{contr}
we     define the
 {\it weighted contraction}   of $S$ as
  the weighted graph $(\Gamma_{/S},w_{/S})$
 where 
  $\GS$  is the  contraction   of   $S$.
 The weight function $\wS$ is defined by setting, for every $\ov{v}\in V(\GS)$,
 \begin{equation}
\label{wS}
 \wS(\ov{v})=b_1(\sigma^{-1}(\ov{v}))+\sum_{v\in \sigma_V^{-1}(\ov{v})}w(v).
  \end{equation}
   \end{defi}
 \begin{example}
Let $S=\{ e\}$. If $e$ is a loop    based at a vertex   $v_0$, 
 let $\ov{v_0}\in  V(\GS)$  be the image of $v_0$, hence the image of  the contracted loop
  $e$. Then  $ \wS(\ov{v_0})=w(v_0)+1$. Whereas  for 
   $\ov{v}\in V(\GS)$ with $\ov{v}\neq \ov{v_0}$ we have 
  $ \wS(\ov{v})=w(v)$.
  If $e$ is not a loop, then  $\wS(\ov{v})=\sum_{v\in \sigma_V^{-1}(\ov{v})}w(v)$
  for every $\ov{v}\in V(\GS)$.
   \end{example} 
    \begin{example}
     \label{contfig}
     In the next picture we have two weighted contractions; the starting graph $(\G,w)$ has all vertices of weight zero,  represented by a ``$\circ$", so that $\G$ has genus 3.
We first contract the non-loop edge $e_1$, so that the weighted contraction has again weight function equal to zero.
Then we contract a loop edge, so that the   weighted contraction has  one vertex  of weight 1, represented by 
 a ``$\bullet$".
 \begin{figure}[h]
\begin{equation*}
\xymatrix@=.5pc{
(\Gamma,w) = &&\ar@{-}@(ul,dl)*{\circ}\ar @{-} @/_.9pc/[rr]  \ar@{-} @/^.9pc/[rr]^{e_1}
&&*{\circ}\ar@{-}@(ur,dr)^{e_2} &&&\ar @{>}[rr] &&& 
&*{\circ} \ar@{-}@(ul,dl)\ar@{-}@(ur,dr)^{e_2}\ar@{-}@(ld,rd)&&&\ar @{>}[rr] &&(\Gamma_{/e_1,e_2},w_{/e_1,e_2}) = &&*{\bullet} \ar@{-}@(ul,dl)\ar@{-}@(ld,rd)
&&&\\
}
\end{equation*}
\end{figure}
   \end{example}

\begin{remark}
\label{genconst}
By the identities (\ref{gendec}) and (\ref{lme}), we have
$$
g(\GS, \wS)=g(\Gamma, w).
$$
\end{remark}
Let $(\Gamma' ,w')$ be a weighted graph.
 We denote
\begin{equation}
\label{min}
(\Gamma,w) \geq (\Gamma' ,w')\  {\text{ if }} \  (\Gamma' ,w'){\text{ is a weighted contraction of  }} (\Gamma,w).
\end{equation}

The next definition generalizes  \cite[Def 3.1.3]{BMV} (with   changes in terminology). 

\begin{defi}
 \label{tropcw}
An {\it  $n$-pointed (weighted) tropical curve} of genus $g$ is a triple  $(\Gamma, w, \ell)$ where $(\Gamma,w)$ is a weighted graph
of genus $g$ with  
legs   $L(\Gamma)=\{x_1,\ldots, x_n\}$, 
$\ell$ a function  
$\ell:E(\Gamma)\cup L(\Gamma) \to \R_{> 0}\cup\{\infty \}$
 such that $\ell(x)=\infty$ if and only if $x$ is a    leg or   an edge adjacent to $1$-valent vertex of weight $0$.

(As in Definition~\ref{troppc}, the legs $\{x_1,\ldots, x_n\}$ are the    marked points.)

If  $w(v)=0$ for every $v\in V(\Gamma)$, we write $w=\mo$ and
$(\Gamma, \underline{0},\ell)$,
 and say that the tropical curve is
{\emph{pure}}, consistently with Definition~\ref{troppc}.

A tropical curve is called {\emph{regular}} if it is pure and if $\Gamma$ is a 3-regular graph
(every vertex has valence $3$).

Two $n$-pointed tropical curves are tropically equivalent if they can be obtained  from one another by adding or removing  $2$-valent vertices of weight $0$, 
or   $1$-valent vertices of weight $0$  together with their adjacent edge.
\end{defi}

The terminology ``pure'' tropical curve is useful to keep track of the original definition, to
connect to the rest of the literature.
The notion of regular tropical  curve, together with referring to the $3$-regularity of its graph,   
suggests   that, in our view, pure tropical curves with a 3-regular graph play the role
of regular (i.e. smooth) curves in the  moduli theory of algebraic curves; in fact,
as we shall see, regular tropical curves have  the following property.

{\it If a tropical curve $C$ specializes to a regular one, then $C$ itself is regular.}

The same holds for nonsingular algebraic curves: if an algebraic curve $X$ specializes to a nonsingular one, then $X$  is nonsingular. See Theorem~\ref{corr}
for more on this point.

Observe, however, that in Section~\ref{secteich} we shall describe a scenario in which the role of smooth algebraic curves is  played by pure tropical curves. 

\begin{example}
(This generalizes Example~\ref{0triv}.)
A graph made of one vertex of weight $g$ and no edges nor legs is a weighted tropical curve
of genus $g$. If $g\leq 1$ such a curve is the specialization of 
the degenerate cases described in Example~\ref{0triv}, and 
 will  thus be excluded by our future numerical assumption, namely  $2g-2+n>0$. 
If $g\geq 2$ these curves play a role and must be considered.
\end{example}

To study moduli of tropical curves we need to generalize Proposition~\ref{rep3},
which is straightforward, once we provide the correct replacement for  graphs having no vertex of valence less than $3$. 
Here is how to do that:
\begin{defi}
\label{stable}
A weighted graph, or a metric weighted graph,  is called {\it stable}\footnote{The terminology is motivated by the subsequent Remark~\ref{stablec}.} if any vertex of weight $0$ has valence at least $3$,
   and any vertex of weight $1$ has valence at least $1$.
   \end{defi}
For instance,    $(\Gamma, \mo)$ is stable if and only if $\Gamma$ has no vertex of valence   $\leq 2$. 
All graphs in Example~\ref{04g} are stable
\begin{remark}
If $(\Gamma,w)$ is stable and  $(\Gamma' ,w')\leq (\Gamma,w)$,  then  $(\Gamma' ,w')$ is stable.
\end{remark}

Now we can  generalize  Proposition~\ref{rep3}. The proof is the same,
provided that 
``$1$-valent vertices" are replaced by ``$1$-valent vertices of weight $0$", and hence
  ``leaves" are replaced by ``edges adjacent to a $1$-valent vertex of weight $0$''.

\begin{prop}
\label{rep3w} Assume $2g-2+n\geq 1$.
\begin{enumerate}
\item
\label{rep3aw}
Every tropical equivalence class of $n$-pointed, weighted tropical curves contains a representative whose $n$ marked points are disjoint.
\item
\label{rep3bw}
The set of tropical equivalence classes of $n$-pointed, weighted tropical curves of genus $g$ 
is in bijection
with the set of stable metric weighted graphs of genus $g$ with $n$ legs.
\end{enumerate}
\end{prop}
The following well known, elementary  fact will be useful.
  \begin{lemma}
\label{max}
Let $(\Gamma, w)$ be a genus $g$ stable graph with $n$ legs.
Then  $|E(\Gamma)|\leq 3g-3+n$ and equality holds if and only if $\Gamma$ is a 3-regular graph with
$b_1(\Gamma)=g$. Moreover, 
in such a case 
 $w=\underline{0}.$ 
\end{lemma}
\begin{proof}
We use induction on $n$. Let us begin with the base case, $n=0$.
We have 
$ 
g=b_1(\Gamma)+|w|=|E(\Gamma)|-|V(\Gamma)|+1+|w|;
$ 
hence,  as $|w|\geq 0$,
$$
|E(\Gamma)|= g-1-|w|+|V(\Gamma)|\leq g-1 +|V(\Gamma)|
$$
and the maximum is achieved for $|w|=0$. In this case $g=b_1(\Gamma)$ and every vertex of $\Gamma$ must have valence at least $3$ (as $(\Gamma ,w)$ is stable),
therefore 
 $$
g=|E(\Gamma)|-|V(\Gamma)|+1 \geq 3|V(\Gamma)|/2- |V(\Gamma)|+1= |V(\Gamma)|/2+1
$$
hence $|V(\Gamma)|\leq 2g-2$ and equality holds if and only if $\Gamma$ is 3-regular. In such a case we have $|E(\Gamma)|=3g-3$ and $b_1(\Gamma)=g$.

Let us assume $n>0$. Let $(\Gamma,w)$ be a stable graph of genus $g$ for which
$|E(\Gamma)|$ is maximum. We can reduce to assume that $\{g,n\}\neq \{0,3\}$
and $\{g,n\}\neq \{1,1\}$ by treating these trivial cases separately.
Denote by $(l_1,\ldots,l_n)$ the legs of $\Gamma$.
Let $\Gamma':=\Gamma -l_n$, hence $\Gamma'$ has $n-1$ legs, and same vertices and   edges as $\Gamma$.
We claim that $\Gamma'$ is not stable. Indeed, if $\Gamma'$ is stable, 
we can construct a new stable graph $\Gamma''$, of genus $g$ as follows. Pick an edge $e'\in E(\Gamma')$, insert a weight-$0$ vertex 
$u$ in its interior, and add a leg adjacent to $u$. This gives a graph, stable of genus $g$ and having a number of edges equal to $|E(\Gamma)|+1$. This contradicts the maximality of  $|E(\Gamma)|$.

So, $\Gamma'$ is not a stable graph; this means that
the vertex $v\in V(\Gamma)$ adjacent to $l_n$ has weight $0$ and 
valence $3$. 
Now look at 
 $v$ as a vertex of $\Gamma'$; it has
valence $2$ and,
 by our reduction, $\Gamma'$ has either
 two distinct edges, or an edge and a leg, adjacent to $v$.
We consider the weighted contraction of $\Gamma'$ given by contracting
an edge adjacent to $v$.
 This is a stable graph, $(\Gamma^*,w^*)$, of genus $g$.
with $n-1$ legs.
By induction we have $|E(\Gamma^*)|=3g-3+n-1$, moreover  $\Gamma^*$ is $3$-regular with weight function constantly zero. But then
$$
|E(\Gamma)|=|E(\Gamma^*)|+1=3g-3+n-1+1=3g-3+n
$$
and $\Gamma$ is $3$-regular by construction. The proof is complete.
\end{proof}

\section{The moduli space of tropical curves}  
\label{modsec}
From now on, we shall assume $2g-2+n \geq 1$ and we shall
consider weighted tropical curves up to tropical equivalence.
Hence, by
Proposition~\ref{rep3w} we can identify our $n$-pointed 
tropical curves with stable metric graphs.

Everything we shall say about weighted tropical curves holds for pure tropical curves,
with obvious modifications.

 \subsection{Tropical curves with fixed combinatorial type}
\begin{defi}
We say that  two  
 tropical curves $(\Gamma,w,\ell)$ and $(\Gamma',w',\ell')$
 with $n$ marked points  $L(\Gamma)=\{x_1,\ldots, x_n\}$ and $L(\Gamma')=\{x_1',\ldots, x_n'\}$
 are
isomorphic,  and write 
  $(\Gamma,w,\ell) \cong (\Gamma',w',\ell')$, if  there is an isomorphism $\alpha$ of the underlying unweighted $n$-pointed curves $(\Gamma,\ell)$ and $(\Gamma',\ell')$
as in Definition~\ref{troppc},   such that 
  $\forall v\in V(\Gamma)$   we have
$ 
w(v)=w'(\alpha(v)).
$ 

We write $\Aut(\Gamma,w,\ell)$ for  the set of automorphisms of a tropical curve $(\Gamma,w,\ell)$.
\end{defi}

Let us fix a stable graph $(\Gamma, w)$  of genus $g$ with  labeled legs $L(\G)=\{x_1,\ldots, x_n\}$, and let us introduce the set $\Mt(\Gamma,w)$
of isomorphism classes of
$n$-pointed tropical curves supported on $(\Gamma,w)$.
In order to study  $\Mt(\Gamma,w)$, we introduce the open cone 
$$R(\Gamma,w):=\R^{E(\Gamma)}_{> 0}.$$ Any element in
$ R(\Gamma,w)$ defines a unique metric graph supported on $(\Gamma, w)$.
Therefore there is a natural surjection 
\begin{equation}
\label{CM}
\pi:R(\Gamma,w)\la \Mt(\Gamma,w);
\end{equation}
it is clear that $\pi(\Gamma, w,\ell)=\pi(\Gamma, w,\ell')$ if and only if $(\Gamma, w,\ell)\cong (\Gamma, w,\ell')$.

The closure of $R(\Gamma,w)$ is, of course, the closed cone
$$
\ov{R(\Gamma,w)}=\R^{ E(\Gamma)}_{\geq 0}\subset \R^{E(\Gamma)}.
$$
 Let $p\in \ov{R(\Gamma,w)}\smallsetminus R(\Gamma,w)$ and, to simplify the notation,
suppose that 
$$p=(t_1,\ldots, t_m,0,\ldots, 0)$$ with $t_i>0$ for all $i=1,\ldots,m$, for some $0\leq m< |E(\Gamma)|.$ 
Let us show that there is a unique $n$-pointed tropical curve $(\Gamma_p, w_p,\ell_p)$ of genus $g$ associated to $p$.
Let 
$$S=\{e_{m+1},\ldots, e_{|E(\Gamma)|}\}\subset E(\Gamma)=\{e_1,\ldots, e_{|E(\Gamma)|}\};$$
then $(\Gamma_p,w_p)=(\GS,\wS)$, i.e.  $(\Gamma_p,w_p)$ is the weighted contraction
of $(\Gamma,w)$ obtained by contracting $S$ (defined in \ref{cont}).  We thus have a natural identification 
 $$E(\Gamma_p)=E(\Gamma)\smallsetminus S=\{e_1,\ldots, e_m\}.
 $$
 The length function
 $\ell_p$ is defined by setting  $\ell_p(e_i)=t_i$ for all $e_i\in E(\Gamma_p)$.
As we noticed   in Remark~\ref{genconst},  we have $g=g(\Gamma_p, w_p)$.

Summarizing: we showed that the boundary points of $\ov{R(\Gamma,w)}$ parametrize $n$-pointed tropical curves  of genus $g$
whose underlying weighted graph is a contraction of $(\Gamma,w)$. More precisely,
for any  
$I\subset\{1,2,\ldots,|E(\Gamma)|\}$,
denote by $\ov{R(\Gamma,w)}_I\subset \ov{R(\Gamma,w)}$
the open face 
$$
\ov{R(\Gamma,w)}_I:=\{(t_1,\ldots, t_{|E(\Gamma)|})\in \ov{R(\Gamma,w)}:\  t_i=0   \  \forall i\in I,\  t_i>0   \  \forall i\not\in I\}.
$$
Next, write  $S_I:=\{e_i,\forall i\in I\}\subset E(\Gamma).$ 
We have proved the following.
\begin{lemma}
\label{cl}
With the above  notation,
the partition
$$
\ov{R(\Gamma,w)}=\bigsqcup_{I\subset \{1,2,\ldots,|E(\Gamma)|\}}\ov{R(\Gamma,w)}_I
$$
is such that for every $I$ there is  a natural isomorphism
$\ov{R(\Gamma,w)}_I\cong R(\Gamma',w')$
where $(\Gamma',w')=(\Gamma_{/S{_I}}, w_{/S{_I}})$.
 \end{lemma}
 \begin{example}
 \label{origin}
If   $I=E(\Gamma)$ then
$\ov{R(\Gamma,w)}_I=\{0\}$, corresponding to the graph with no edges,  one vertex
of weight $g$, and $n$ legs. It is clear that this graph can be obtained as weighted contraction from every genus $g$ weighted graph.
 \end{example}

Consider two points $p_1, p_2\in \ov{R(\Gamma,w)}$.
By what we said  there exist two metric weighted graphs associated to them, denoted
$(\Gamma_1,w_1,\ell_1)$ and $(\Gamma_2,w_2,\ell_2)$.
We have  an equivalence on $\ov{R(\Gamma,w)}$:
\begin{equation}
\label{sim}
p_1\sim p_2   \  \ \text{ if  }\   \ (\Gamma_1,w_1,\ell_1)\cong (\Gamma_2,w_2,\ell_2).
\end{equation}
The quotient with respect to this equivalence relation will be denoted
\begin{equation}
\label{MG}
\ov{\pi}:\ov{R(\Gamma,w)}\la \ov{\Mt(\Gamma,w)}:=\ov{R(\Gamma,w)}/\sim 
\end{equation}
and   $\ov{\Mt(\Gamma,w)}$ is a topological space, with the quotient topology of the euclidean topology on $\ov{R(\Gamma,w)}= \R^{E(\Gamma)}_{\geq 0}$. 
A precise description of the fibers of $\ov{\pi}$ is given in Lemma~\ref{preimage} below.

Now we must study  the automorphisms of $(\G,w)$,  and how  they act $\ov{R(\Gamma,w)}$.

 \begin{defi}
Let  $(\Gamma,w) $ be a weighted graph  with    labeled legs $L(\G)=\{x_1,\ldots, x_n\}$.
An automorphism of $(\Gamma,w) $  
is an isomorphism,   $\alpha$, of $\G$ with itself as defined in \ref{combg},
such that $\alpha(x_i)=x_i$ for every $i=1,\ldots, n$, and 
$ 
w(v)=w(\alpha(v))$ for every $v\in V(\Gamma)$.
We denote by $\Aut(\Gamma,w)$ the group of automorphisms of $(\Gamma,w)$.
\end{defi}

\begin{remark}
\label{aut}
If $(\Gamma,w,\ell)$ is a pointed curve, then  $\Aut(\Gamma,w,\ell)\subset \Aut(\Gamma,w)$.
The automorphism group of a weighted graph, and of a tropical curve, is finite.
\end{remark}
Of course,  $\Aut(\Gamma,w)$ 
is made of   pairs $\alpha=(\alpha_V,\alpha_E)$
of permutations on the vertices and the edges   satisfying 
some compatibility conditions; in fact, by definition, $\alpha$ has to fix the legs
(i.e. with the notation of Definition~\ref{combg}, $\alpha_L=id_{L(\G)}$).
Therefore   $\Aut(\Gamma,w)$
acts on $\ov{R(\Gamma,w)}$ by permuting the coordinates according to
$\alpha_E$, in particular, 
 $\Aut(\Gamma,w)$ acts as group of isometries.

$\Aut(\Gamma,w)$ may contain   non-trivial elements   acting trivially on $\ov{R(\Gamma,w)}$.

\begin{example}
The loop-inversion automorphism described in Example~\ref{invertloop} acts trivially on $\ov{R(\Gamma,w)}$.
\end{example}
\begin{example}
Assume $L(\G)=\emptyset$ and $V(\G)=\{v_1,v_2\}$
with $w(v_1)=w(v_2)$; suppose $E(\G)=\{e_1,\ldots, e_{n}\}$ and let
  $\G$ have no loops, as in the picture below. So $(\G,w)$ has genus $n-1+|w|$. 
   \
   
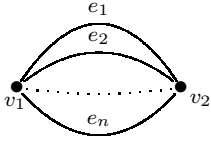
    \begin{figure}[h]
\begin{equation*}
\xymatrix@=.5pc{
 &\\
  &&*{\bullet}
  \ar @{.} @/_.2pc/[rrrrr]_(0.01){v_1} \ar @{-} @/_1.5pc/[rrrrr]^{e_n} _(1){v_2}\ar@{-} @/^1.1pc/[rrrrr]^{e_2}\ar@{-} @/^2pc/[rrrrr]^{e_1}
&&&&& *{\bullet} &&&&&&&
\\
 &\\
}
\end{equation*}
\end{figure}
\

  Now,
  $\G$ has an involution  swapping $v_1$ and $v_2$,
  and   every conjugate pair of half-edges $\{h, \ov{h}\}$ (so that the edges are all fixed).
This automorphism acts trivially on $\ov{R(\Gamma,w)}$.
 
 Denote by
 ${\mathcal S}_n$ the symmetric group, then
  $$\Aut (\G,w)\cong {\mathcal S}_n\times \Z/2\Z
  $$
  where the ${\mathcal S}_n$ factor accounts for   automorphisms permuting the edges,
  which clearly act non-trivially on $\ov{R(\Gamma,w)}$.

  Observe that if   $L(\G)\neq \emptyset$, or  if $w(v_1)\neq w(v_2)$, then 
  $\Aut (\G,w)\cong {\mathcal S}_n$.

\end{example}

In the sequel we simplify the notation and  set, for $I\subset (1,\ldots, |E(\Gamma)|)$,
\begin{equation}
\label{simpnot}
F_I:=\ov{R(\Gamma,w)}_I,\quad \text{and} \quad G_I:=\Aut(\Gamma_{/S_I}, w_{/S_I})
\end{equation}
so that $G_I$ acts on  $\ov{F_I}$ (closure of $F_I$ in $R(\Gamma,w)$) by permuting the coordinates.
As $F_I$ varies among the (open) faces of $\ov{R(\Gamma,w)}$ it may very well happen that 
different faces correspond to
isomorphic   weighted graphs. Let us introduce some notation to keep track of this fact.
For a fixed $I\subset  \{1,\ldots, |E(\Gamma)|\}$, 
 we denote  
 $$
 \Is(I):=\{J\subset (1,\ldots, |E(\Gamma)|): \ \   (\Gamma_{/S{_J}}, w_{/S{_J}})\cong (\Gamma_{/S{_I}}, w_{/S{_I}})\}.
 $$
Next,
  for every $J\in \Is(I)$ we    fix an 
 isomorphism $\Phi_J:(\Gamma_{/S{_I}}, w_{/S{_I}})\to (\Gamma_{/S{_J}}, w_{/S{_J}})$, and the isometry
 $$
 \phi_J: {\ov{F_I}}\stackrel{\cong}{\la}{\ov{F_J}}
 $$ induced by $\Phi_J$. Notice that $\phi_J$ is induced by a bijection  between the natural coordinates of $F_I$ and $F_J$.
  If $I=J$ we shall assume that $\Phi_I$ is the identity.
 For every point $p\in F_I$ we denote
 $ 
 \Is(p):=\Is (I).
 $ 
\begin{lemma}
\label{preimage}
Let $p\in \ov{R(\Gamma,w)}$.
Then, with the notation \eqref{MG},
$$
 \ov{\pi}^{-1} (\ov{\pi}(p))=\{g\phi_J(p),\  \forall J\in \Is(p),\  \forall g\in   G_J\}.
 $$
\end{lemma}
  \begin{proof}
The inclusion  $ \ov{\pi}^{-1} (\ov{\pi}(p))\supset \{g\phi_J(p),\  \forall J\in \Is(p),\  \forall g\in   G_J\}$
is obvious, as the set on the right parametrizes isomorphic metric weighted graphs.
  
Let $F_I$ be the face containing $p$.
 Let $r\in \ov{R(\Gamma,w)}$ be such that $\ov{\pi}(p)=\ov{\pi}(r)$. 
We have an isomorphism  $(\Gamma_p,w_p,\ell_p)\cong (\Gamma_r,w_r,\ell_r)$, hence
 $F_I\cong R(\Gamma_p,w_p)\cong  R(\Gamma_r,w_r)$.
 Suppose first that $r\in F_I$. Then the underlying weighted graph of $p$ and $r$ is the same,
 namely $(\Gamma_{/S{_I}},w_{/S{_I}})$,
 and an  isomorphism   between them is an element of    $\Aut(\Gamma_{/S{_I}},w_{/S{_I}})=G_I$
preserving the lengths of the edges.
In other words, there exists $g\in G_I$ such that $r=gp$ (recall that  $\phi_I=Id _{F_I}$).

Now let  
  $F_{J}\cong R(\Gamma_r,w_r)$ be the face containing $r$, with $J\neq I$. Of course, $J\in \Is(p)$, hence
 we have an isometry
 $ 
 \phi_J:F_{I} \to F_J,
  $ 
 induced by an isomorphism between the underlying graphs.
 It is clear that $r':=\phi_J(p)$ parametrizes a metric weighted graph isomorphic to the one parametrized by $p$.
Therefore  $\ov{\pi}(r')=\ov{\pi}(r)$, hence, by the previous part,  there exists $g\in G_J$ such that 
$
r=gr'=g\phi_J(p).
$
\end{proof}

For any group $G$ acting on $ \ov{R(\Gamma,w)}$ and any
subset $Z\subset \ov{R(\Gamma,w)}$, we denote by $Z^G\subset \ov{R(\Gamma,w)}$ the union of the $G$-orbits of the elements in $Z$.

\begin{prop}
\label{ovM}
Let $(\Gamma,w)$  be a stable graph.
\begin{enumerate}
\item
\label{ovMc} There is a canonical decomposition (notation in (\ref{min}))
$$
\ov{\Mt(\Gamma,w)}=\bigsqcup_{(\Gamma',w')\leq (\Gamma,w)} \Mt(\Gamma',w'),
$$
where $\Mt(\Gamma,w)$ is   open and dense  in $\ov{\Mt(\Gamma,w)}$.
\item
\label{quot}
The quotient map $\ov{\pi}:\ov{R(\Gamma,w)}\to \ov{\Mt(\Gamma,w)}$ factors as follows:
$$
\ov{\pi}:\ov{R(\Gamma,w)}\stackrel{\tau}{\la} \ov{R(\Gamma,w)}/\Aut(\Gamma,w)\stackrel{\gamma}{\la}  \ov{\Mt(\Gamma,w)}. 
$$
Moreover  $\tau$ is   open  and $\ov{\pi}$ has finite fibers.
\item$\ov{\Mt(\Gamma,w)}$ is a Hausdorff topological space.
 \end{enumerate} 
\end{prop}
\begin{proof}
The existence of the decomposition is a straightforward consequence of the definition 
of  $\ov{\Mt(\Gamma,w)}$ and of Lemma~\ref{cl}.

Recall from (\ref{CM}) that $\Mt(\Gamma,w)=R(\Gamma,w)/\sim$.
Now 
$R(\Gamma,w)$ is open and dense in $\ov{R(\Gamma,w)}$ and we have
$\ov{\pi}^{-1}(\Mt(\Gamma,w))=R(\Gamma,w)$.
Therefore 
$\Mt(\Gamma,w)$ is open and dense in $ \ov{\Mt(\Gamma,w)}$.

Now let $G:=\Aut(\Gamma,w)$.
The restriction of $\ov{\pi}$ to $R(\Gamma,w)$ is the quotient
  $\pi:R(\Gamma,w)\to \Mt(\Gamma,w)=R(\Gamma,w)/G$, 
  so in this case part (\ref{quot}) is proved.
  
 Let now $p\in \ov{R(\Gamma,w)}_I$ with $I\neq \emptyset$ (cf. Lemma~\ref{cl}).
 Then $p=(t_1,\ldots, t_{|E(\Gamma)|})$ with $t_i=0$ for every $i\in I$, and $t_i>0$ otherwise.
 Let $\alpha\in G$, then   $\alpha$ acts as a permutation on $\{1,\ldots,|E(\Gamma)|\}$,
hence  
 $$
 \alpha (p)=(t_{\alpha^{-1}(1)},\ldots, t_{\alpha^{-1}(|E(\Gamma)|)})\in \ov{R(\Gamma,w)}_{\alpha(I)},
 $$
where $\alpha(I)=\{\alpha(i),\ \forall i\in I\}\subset\{1,2,\ldots, |E(\Gamma)|\}$.
Let $S=S_I\subset E(\Gamma)$ and  $\alpha(S)\subset  E(\Gamma)$  be the set of edges corresponding, respectively,  to $I$ and $\alpha(I)$.
 Lemma~\ref{cl} and the discussion preceding it yield
$$
 \ov{R(\Gamma,w)}_I=R(\GS, \wS),\  \  \  \   \ov{R(\Gamma,w)}_{\alpha(I)}=R(\Gamma_{/\alpha(S)}, w_{/\alpha(S)}).
$$
The automorphism $\alpha$ maps $S$ into $\alpha(S)$, hence it induces an isomorphism
$$
\ov{\alpha}:\GS \stackrel{\cong}{\la}\Gamma/\alpha(S).
$$
Moreover, $\alpha$ (as any automorphism of any graph) maps bijectively cycles to cycles, therefore $\ov{\alpha}$ induces an isomorphism $(\GS, \wS)\cong (\Gamma/\alpha(S), w/\alpha(S))$.
Finally, the point $p$ corresponds to a metric graph on $(\GS, \wS)$ with length function
$\ell(e_i)=t_i$ for every $i\not\in I$ (we have $E(\GS)=E(\Gamma)\smallsetminus S$);
it is clear that
 $\alpha(p)$ corresponds to a metric graph on $(\GS, \wS)$ with length  $\alpha(\ell)$ given by 
$\alpha(\ell)(e_{\alpha(i)})=t_{\alpha^{-1}(\alpha(i))}=t_i$ for every $i\not\in I$.
But then $p$ and $\alpha(p)$ parametrize isomorphic metric weighted graphs. Therefore $\ov{\pi}(p)=\ov{\pi}(\alpha(p))$ and the factorization of part (\ref{quot}) is proved.

The fact that $\tau$ is open follows easily from   $G$ being a finite group of isometries of $ \ov{R(\Gamma,w)}$.
Similarly, we easily  get that $\ov{R(\Gamma,w)}/\Aut(\Gamma,w)$ is Hausdorff.

 The fibers of ${\ov{\pi}}$ are finite by lemma~\ref{preimage}.

It remains to  show that $\ov{\Mt(\Gamma,w)}$ is Hausdorff; 
to do that we will use the notation (\ref{simpnot}).
Let $\ov{p}$ and  $\ov{q}$ be two distinct points in $\ov{\Mt(\Gamma,w)}$.
Set $\ov{\pi}^{-1}(\ov{p})=\{p_1,\ldots, p_m\}$ and $\ov{\pi}^{-1}(\ov{q})=\{q_1,\ldots, q_n\}$.
There exists an $\epsilon>0$ small enough  that the following holds.

\noindent
(1) For every $i,j$ the open balls $B_{p_i}(\epsilon)$ and $B_{q_j}(\epsilon)$ do not intersect each other.

\noindent
(2) If $\ov{F_I}\cap B_{p_i}(\epsilon)\neq \emptyset$ then $p_i\in \ov{F_I}$. Similarly,
if 
 $\ov{F_I}\cap B_{q_j}(\epsilon)\neq \emptyset$ then $q_j\in \ov{F_I}$.
 
 We now set 
 $$
 U:=\cup_{i=1}^m B_{p_i}(\epsilon)\cap  \ov{R(\Gamma,w)}\quad \text{and} \quad V:=\cup_{j=1}^n B_{q_j}(\epsilon) \cap \ov{R(\Gamma,w)}.
 $$

It is clear that $U$ and $V$ are open subsets of $\ov{R(\Gamma,w)}$; moreover 
$U\cap V=\emptyset$, by (1) above.
We claim that 
\begin{equation}
\label{sat}
\ov{\pi}^{-1} (\ov{\pi}(U))=U \quad \text{and} \quad \ov{\pi}^{-1} (\ov{\pi}(V))=V.
 \end{equation}
 
 To prove it  we set some simplifying conventions.
Pick $I\subset \{1,\ldots, |E(\Gamma)|\}$ and consider the group $G_I$, which acts on $ \ov{F_I}$.
Consider also, for every  $J\in \Is (I)$, the isometry $\phi_J:\ov{F_I}\to \ov{F_J}$
(cf. Lemma~\ref{preimage}).
We extend  to   $\ov{R(\Gamma,w)}$ the action of $G_I$
and the map $\phi_J$ as follows.
For every $p\in \ov{R(\Gamma,w)}\smallsetminus \ov{F_I}$,  every $g\in G_I$ and every $\phi_J$
as above, 
 we set $gp=p$ and $\phi_J(p)=p$.
Now,  by Lemma~\ref{preimage}  
to prove the claim it suffices to prove the following two facts.

\noindent (a) For every $I$ and every $g\in G_I$ we have
 $U^g\subset U$.  
 
 \noindent (b) For every $u\in U$ and every $J\in \Is (u)$ we have $\phi_J(U)\subset U$.

 Pick $I$ and $g\in G_I$; as $g$ acts as the identity away from $\ov{F_I}$  we can assume that
 $\ov{F_I} \cap U \neq \emptyset$, i.e. that there exists $p_i$ such that
 $
 \ov{F_I}\cap B_{p_i}(\epsilon)\neq \emptyset.
 $
 By (2) above, this implies that $p_i\in \ov{F_i}$; hence $gp_i\in \pi^{-1}(\ov{p})$ and
 $$
  (\ov{F_I}\cap B_{p_i}(\epsilon))^g=\ov{F_I}\cap B_{gp_i}(\epsilon)\subset  U 
 $$
 (as $G_I$ preserves the metric of $\ov{F_i}$). 
 This proves (a).
To prove (b), let $u\in U\cap F_I$, let $J\in \Is(u)$ and let $\phi_J:\ov{F_I}\to \ov{F_J}$.
As before, there exists $p_i$ such that
 $ 
 u\in \ov{F_I}\cap B_{p_i}(\epsilon).
 $ 
Therefore
$$
\phi_J(\ov{F_I}\cap B_{p_i}(\epsilon))=\ov{F_J}\cap B_{\phi_J(p_i)}(\epsilon).
$$
By what we proved before, $\ov{\pi}(p_i)=\ov{\pi}(\phi_J(p_i))$, hence 
$\ov{F_J}\cap B_{\phi_J(p_i)}(\epsilon)\subset U$. Now (b) is proved, and claim (\ref{sat}) with it.
 This yields that $\ov{\pi}(U)$ and $\ov{\pi}(V)$ are open and disjoint in 
 $\ov{\Mt(\Gamma,w)}$. Since obviously 
 $\ov{p}\in \ov{\pi}(U)$ and  $\ov{q}\in \ov{\pi}(V)$ 
 we are done.
\end{proof}
\begin{remark}
The map $\gamma$ in Proposition~\ref{ovM} (\ref{quot})
identifies isomorphic curves which are not identified by the automorphisms of $(\G,w)$.
The point is: in general,
$\Aut (\Gamma,w)$ does not induce every automorphism of its weighted contractions,
nor does it induces all the isomorphisms between its weighted contractions.


For instance, consider  the graph $(\Gamma,w)$ in Example~\ref{contfig}.
Then the graph in the middle of the picture, $(\Gamma_{/e_1}, w_{/e_1})$,
has an ${\mathcal S}_3$ of automorphisms permuting its three loops; these automorphisms act non trivially on
$\ov{R(\G,w)}$, and are not all induced by automorphisms of $(\G,w)$.

Next, contracting any one of the three loops of $(\Gamma_{/e_1}, w_{/e_1})$ gives three isomorphic
graphs (one of which is drawn at the right of that picture). 
It is clear that $\Aut (\Gamma,w)$ does not act transitively on them.
\end{remark}
\subsection{Construction and properties of $\Mgnt$}
We shall now construct the moduli space of $n$-pointed tropical curves of genus $g$, $\Mgnt$, as a topological space.
Recall that we always assume $2g-2+n\geq 1$.
We set (denoting by
 ``$\cong$" isomorphism of tropical curves)
\begin{equation}
\label{Mg}
\Mgtnn:= \Bigl(\  \bigsqcup_{\stackrel{\Gamma \  3-{\text{regular}}}{b_1(\Gamma)=g, |L(\Gamma)|=n}}\ov{\Mt(\Gamma,\underline{0})}\Bigr)/\cong.
\end{equation}

\begin{remark}
\label{qgr}
Consider the quotient map
\begin{equation}
\label{qg}
\pi_g: \bigsqcup_{\stackrel{\Gamma \  3-{\text{regular}}}{b_1(\Gamma)=g,|L(\Gamma)|=n}}\ov{\Mt(\Gamma,\underline{0})} \la \Mgtnn .
\end{equation}
If $\Gamma$ is    3-regular  with $b_1(\Gamma)=g$,
then $\pi_g^{-1}(\pi_g([(\Gamma,\underline{0},\ell )])) = [(\Gamma,\underline{0}, \ell)]$.
\end{remark}
Comparing with (\ref{MG}) we have another description of $\Mgtn$:
\begin{equation}
\label{Mgg}
\Mgtnn= \Bigl(\  \bigsqcup_{\stackrel{\Gamma \  3-{\text{regular}}}{b_1(\Gamma)=g,|L(\Gamma)|=n}}\ov{R(\Gamma,\underline{0})}\Bigr)/\sim .
\end{equation}
$\Mgtnn$ is defined  as the topological quotient space of a topological space; the two expressions of $\Mgtnn$ as a quotient clearly yield the same topology.
\begin{remark}
$\Mgnt$ {\it is connected.}
Indeed,   every    $\ov{\Mt(\Gamma,\underline{0})}$ appearing in (\ref{Mg})
is connected and
contains the point parametrizing the metric weighted graph with no edges and one vertex of weight $g$; see Example~\ref{origin}. 
We shall prove later, in Proposition~\ref{connt1}, that $\Mgnt$ is actually connected through codimension one.

\end{remark}

  \begin{thm}
  \label{Mgt}
Assume $2g-2+n\geq 1$.
  \begin{enumerate}
\item
The points of $\Mgtnn$ bijectively parametrize isomorphism classes of $n$-pointed tropical curves of genus $g$ (up to tropical equivalence).
\item
  \label{Mgtr}
Let $\Mgrnn\subset \Mgtnn$ be the subset  parametrizing regular curves, i.e.
 $$
 \Mgrnn =\bigsqcup_{\stackrel{\Gamma \  3-{\text{regular}}}{b_1(\Gamma)=g,|L(\Gamma)|=n}}{\Mt(\Gamma,w)}\subset \Mgtnn,
 $$
and 
$\Mgrnn$   is  open and dense.

\item
  \label{Mgtp}
 Let $\Mgpnn$ be the subset  parametrizing pure tropical curves.  
Then  $\Mgpnn$   is   open and dense.

\item
  \label{MgtH}
 $\Mgtnn$ is    Hausdorff.
 \end{enumerate}
\end{thm}
  
\begin{proof}
Let  $(\Gamma, w,\ell)$ be a stable metric graph of genus $g$ with $n$ legs,
and let us prove that its isomorphism class corresponds to a point in $\Mgtnn$.
This amounts to showing that $(\Gamma, w,\ell)$ is in the closure of $R(\Gamma_0,\underline{0})$ for some 3-regular graph
$\Gamma_0$   with $b_1(\Gamma_0)=g$.
In other words, by Lemma~\ref{cl},
we must prove that $(\Gamma,w)$ is a weighted contraction of $(\Gamma_0,\underline{0})$,  with $\Gamma_0$ a 3-regular graph.

Suppose first that $w=\underline{0}$. A proof of this fact is the proof of \cite[Prop. A.2.4] {CV}(that Proposition is concerned with 3-edge-connected curves,
but the proof is easily seen to work  in this case).

 We continue by induction on $|w|$, the basis being the case $w=\underline{0}$, which we just did.
 Let us now assume that $w(v_1)\geq 1$ for some $v_1\in V(\Gamma)$.
 Let $(\Gamma', w' )$ be the genus $g$ stable graph defined as follows:
 $\Gamma'$ is obtained from $\Gamma$ by just adding a loop, $e_0$, based at $v_1$.
 Therefore we have
 $V(\Gamma')=V(\Gamma)$ and $E(\Gamma')=E(\Gamma)\cup\{e_0\}$.
 Now let
$$w'(v)=\begin{cases}
w(v) & \text{ if $v\neq v_1$}   \\
w(v_1)-1 & \text{ otherwise.}
\end{cases}
$$
  It is clear that $(\Gamma, w)$ is a weighted contraction of $(\Gamma',w')$,
indeed 
$$
 (\Gamma, w) =(\Gamma'/e_0,w'/e_0).
$$
We can apply induction, as $|w'|=|w|-1$; hence   $(\Gamma',w')$ is a weighted contraction of $(\Gamma_0,\underline{0})$
  for some 3-regular graph
$\Gamma_0$, i.e
$ 
(\Gamma',w')=((\Gamma_0)_{/S},\underline{0}_{/S})
$ 
for some $S\subset E(\Gamma_0)$.
 Now,  $e_0\in E(\Gamma')\subset E(\Gamma_0)$, hence, denoting $S_0=\{e_0\}\cup S$, we have
$ 
(\Gamma ,w )=(\Gamma_0 \;_{/S_0},\underline{0}_{/S_0}).
$ 
The first part is proved.

Let us  prove (\ref{Mgtr}); 
the description of $\Mgrnn$ follows from Lemma~\ref{max}. Now
fix a 3-regular graph $\Gamma$;
by Proposition~\ref{ovM} we know that $\Mt(\Gamma,\underline{0})$ is open and dense in $ \ov{\Mt(\Gamma,\underline{0})}$. 
Consider the definition of $\Mgtnn$ given in (\ref{Mg}).
Pick a point  $[(\Gamma, \underline{0},\ell)]\in \Mt(\Gamma,\underline{0})$;  by Remark~\ref{qgr},  
 the map $\pi_g$ induces a homeomorphism of 
$\Mt(\Gamma,\underline{0})$ with its image such that 
$$
\pi_g^{-1}(\pi_g(\Mt(\Gamma,\underline{0})) =\Mt(\Gamma,\underline{0}).
$$

Therefore $\Mt(\Gamma,\underline{0})$ is open in $\Mgtnn$. 
Moreover, the union of the $\Mt(\Gamma,\underline{0})$  as $\Gamma$ runs through 
all 3-regular graphs with $b_1(\Gamma)=g$ is obviously dense in $\Mgtnn$.
Therefore the above union is
open and dense.

Now  (\ref{Mgtp}). 
  We have, of course
  $$
  \Mgrnn\subset \Mgpnn\subset \Mgtnn.
  $$
Hence $\Mgpnn$   is dense by part ~(\ref{Mgtr}).
Now recall that
$$
\Mgpnn=\{[(\Gamma,w,\ell)]: b_1(\Gamma)=g\}=\{[(\Gamma,w,\ell)]: |w|=0\}.
$$
Let us denote by
$ 
R(\Gamma,w)^+\subset \ov{R(\Gamma,w)}
$ the union of all  loci  corresponding to   weighted graphs $(\Gamma',w')$
such that $b_1(\Gamma')<g$ (cf. Lemma~\ref{cl}). The set $R(\Gamma,w)^+$ is closed,
as the first Betti number does not grow under   edge contraction.
Hence its complement, the locus  parametrizing pure tropical curve, is open. Hence
the locus in $ \ov{\Mt(\Gamma,\underline{0})}$ corresponding to pure tropical curves is also open.

Let now ${\ov{p}}\in \Mgpnn\subset\Mgtnn$
and let $(\Gamma_p, \underline{0})$ be its supporting graph. For every 3-regular graph $\Gamma$ with $b_1(\Gamma)=g$
let $p_{\Gamma}\in \ov{\Mt(\Gamma,\underline{0})}$ be the preimage of $p$.
By what we just said $p_{\Gamma}$ admits an open neighborhood $U_{\Gamma}\subset \ov{\Mt(\Gamma,\underline{0})}$
such that $U_{\Gamma}$ parametrizes only pure tropical curves.
Up to shrinking each $U_{\Gamma}$ around $p_{\Gamma}$ we can assume that
$\pi_g (U_{\Gamma})=\pi_g (U_{\Gamma'})$ for all such $\Gamma$ and $\Gamma'$.
But then 
$$
\pi_g(\cup U_{\Gamma})\subset \Mgpnn,\  \  \text{and}\  \  \pi_g^{-1}(\pi_g(\cup U_{\Gamma}))=\cup U_{\Gamma}
$$
where the union is over all 3-regular graphs $\Gamma$ with $b_1(\Gamma)=g$. This implies that
$\pi_g(\cup U_{\Gamma})$ is an open neighborhood of  ${\ov{p}}$ all contained in $\Mgpnn$.
Part (\ref{Mgtp}) is proved.

Finally, let us show that $\Mgtn$ is   Hausdorff.
The quotient map  $\pi_g$
(\ref{qg}) induces a bijection of $\ov{\Mt(\Gamma,\underline{0})}$ with its image.
Therefore, by Proposition~\ref{ovM}, $\Mgtn$ is obtained by gluing together finitely many Hausdorff spaces.
Let $\ov{p},\ov{q}\in \Mgtn$;   
for every 3-regular $\Gamma$ with $b_1(\Gamma)=g$
let $p_{\Gamma}=
\ov{\Mt(\Gamma,\underline{0})}\cap \pi_g^{-1}(\ov{p})$ 
and $q_{\Gamma}= 
\ov{\Mt(\Gamma,\underline{0})}\cap \pi_g^{-1}(\ov{q})$ 
and pick, in  the Hausdorff space $\ov{\Mt(\Gamma,\underline{0})}$, disjoint open neighborhoods of the two points:
$p_{\Gamma}\in U_{\Gamma},q_{\Gamma}\in V_{\Gamma}$.
If $\pi_g^{-1}(\ov{p})$ doesn't  intersect $\ov{\Mt(\Gamma,\underline{0})}$ we don't do anything, similarly for
$\ov{q}$. 
Arguing as before, we can assume that $\pi_g(U_{\Gamma})=\pi_g(U_{\Gamma'})$
for all $\Gamma$ and $\Gamma'$ as above, hence 
  $\pi_g(U_{\Gamma})$ is open; similarly for $V_{\Gamma}$.
Then,   $\pi_g(U_{\Gamma})$ and $\pi_g(V_{\Gamma})$ are open disjoint neighborhoods of
$\ov{p}$ and $\ov{q}$ in $\Mgtn$.
\end{proof}
\begin{remark}
\label{dim}
For every stable graph $(\Gamma,w)$ as above, the space $\Mt(\Gamma, w)$
is  the quotient of the topological manifold $R(\Gamma, w)$
by the finite group $\Aut(\Gamma,w)$. Its dimension can thus be defined as follows:
$$
\dim \Mt(\Gamma, w):=\dim R(\Gamma, w)=|E(\Gamma)|.
$$

More generally, let $X$ be a topological space   containing a dense open subset which is an orbifold
 of dimension $n$ (locally the quotient of an $n$-dimensional topological manifold by a finite group); then  we say that $X$ {\it has pure dimension} $n$. 
\end{remark}
\begin{remark}
\label{puremt}
{\it $\Mgtn$ has pure dimension equal to $3g-3+n$.}

Indeed, by Theorem~\ref{Mgt}   the moduli space of regular curves, $\Mgrnn$,
is open and dense in $\Mgtn$. Now, $\Mgrnn$ is the  disjoint union of finitely many
spaces of type $\Mt(\Gamma, w)$.
By what we just observed, each of these spaces has dimension $|E(\Gamma)|$;
 by Lemma~\ref{max} we have $|E(\Gamma)|=3g-3+n$.
\end{remark}

\begin{example}
Suppose $g=1$, hence $n\geq 1$.
The case $n=1$ is very simple: $\Mt_{1,1}$ is homeomorphic to $\R_{\geq 0}$
with the point $0$ identified with the curve with no edges and a vertex of weight $1$
(see Figure~\ref{11fig}). 
In the next pictures
we list all the combinatorial types in cases $n=1$ and $n=3$; the $0$-weight vertices are 
pictured as ``$\circ$", whereas   the $1$-weight vertices 
are pictured as ``$\bullet$", together  with a ``$+1$'' next to them.
The dashed arrows represent specializations, i.e. weighted contractions.
 \begin{figure}[h]
\begin{equation*}
\xymatrix@=.5pc{
&&&&&&&&&&&&&\\
&&*{\circ} \ar@{-}@(ul,dl) \ar@{-}[r] &&&&\ar @{-->}[rrrr]&&&&&&& &*{\bullet} \ar@{-}[r]^<{+1} &&\\
&&  \dim 1 &&&&&&&&&&&& \dim 0\\
}
\end{equation*}
\caption{The case $g=1$, $n=1$}\label{11fig}
\end{figure}
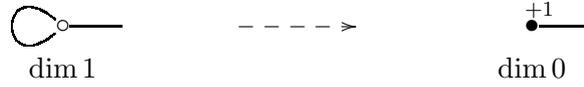

\newpage
Next is the case $n=3$.
It is clear that $\Mt_{1,3}$ is not a manifold.
In the picture the notation ``$\times 2$", or ``$\times 3$", next to a specialization arrow means that the specialization is obtained in two, or three, different ways, i.e. by contracting two, or three, different edges.  

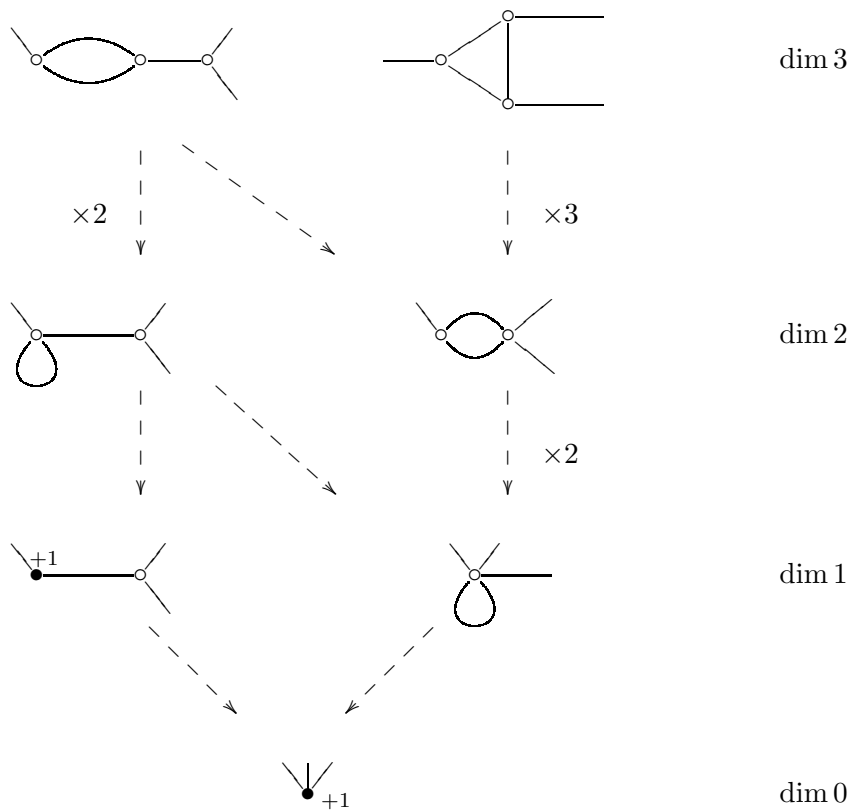
\begin{figure}[h]
\label{1fig}
\begin{equation*}
\xymatrix@=.5pc{
&&&&&&&&&&&&&&*{\circ}\ar@{-}[dll]\ar@{-}[rr]&&\\
 &*{\circ} \ar@{-}[ul]\ar @{-} @/_.7pc/[rr] \ar@{-} @/^.7pc/[rr]
&& *{\circ}\ar@{-}[rr]   & & *{\circ}\ar@{-}[dr]\ar@{-}[ur]
& &&&&&&
*{\circ}\ar@{-}[ll] 
&&&&&&&&&\dim 3 \\
&&&&&&&&&&&&&&*{\circ}\ar@{-}[ull]\ar@{-}[uu] \ar@{-}[rr]&&\\
 &&& \ar @{-->}[ddd]&\ar @{-->}[dddrrrrr] &  &&&&&&&&& \ar @{-->}[ddd]&
\\
\\
&&{\times 2}&&&&&&&&&&&&&{\times 3}\\
&&&&&&&&&&&&&&&&&\\
&&&&&&&&&&&&&&&&&\\
&*{\circ}  \ar@{-}[ul]\ar@{-}@(dr,dl) \ar@{-}[rr]&& *{\circ}\ar@{-}[dr]\ar@{-}[ur]&&&&&& &&
 &*{\circ} \ar@{-}[ul]\ar @{-} @/_.7pc/[rr] \ar@{-} @/^.7pc/[rr]
&& *{\circ}\ar@{-}[ur] \ar@{-}[dr]&&&&&&&\dim 2 \\
&&& \ar @{-->}[ddd] && \ar @{-->}[dddrrrr] &&&&&&&&& \ar @{-->}[ddd]
 &&&&&&& \\
&&&&&&&&&&&&&&&\\
&&&&&&&&&&&&&&&{\times 2}\\
&&&&&&&&&&&&&&&&&\\
&&&&&&&&&&&&&&&&&&&\\
&*{\bullet}\ar@{-}[ul] \ar@{-}[rr]^<{+1}&& *{\circ}\ar@{-}[ur] \ar@{-}[dr]&&&&&&&&&&
 *{\circ}\ar@{-}[ul]\ar@{-}@(dr,dl) \ar@{-}[ur] \ar@{-}[rr] &&&&&&&&\dim 1\\
&&& \ar @{-->}[dddrrr]  &&&&&&&&&  \ar @{-->}[dddlll]
\\
&&&&&&&&&&&&&\\
&&&&&&&&&&&&&\\
&&&&&&&&&&&&&\\
&&&&&&&&&&&&&\\
&&&&&&&& *{\bullet}\ar@{-}[ur]_<{+1}  \ar@{-}[ul] \ar@{-}[u]&&&&&&&&&&&&&\dim 0
}
\end{equation*}
\caption{The case $g=1$, $n=3$} 
\end{figure}
\end{example}
\subsection{Compactification of the moduli space of tropical curves}
\label{comp}
In this subsection we describe a natural compactification of the moduli space of tropical curves,
following an  indication  of G.Mikhalkin 
(see \cite{MIK5} for example).
\begin{defi}
An {\it extended, $n$-pointed tropical curve} (up to tropical equivalence)
of genus $g$ will be a 
triple $(\Gamma, w,\ell)$  where
$(\Gamma, w)$  is a stable graph of genus $g$ with $n$ labeled legs,
and   
$$
\ell:E(\Gamma)\cup L(\Gamma)\to \R_{> 0}\cup\{ \infty\}
$$
is  a     length  function on the edges such that $\ell(x)=\infty$ for every $x\in L(\Gamma)$. 
\end{defi}
We are thus generalizing  the previous definition by   allowing the length of any edge to be infinite.

\begin{remark}
In the sequel, we view $\R\cup\{ \infty\}$ as a topological space with the Alexandroff one-point topology (i.e. the open sets are all the usual open subsets of $\R$, and
all the complements  of compact subsets of $\R$).
Therefore, for every $n$ the spaces
$(\R\cup\{ \infty\})^n$ are compact, Hausdorff spaces.
Hence  also the spaces 
 $(\R_{\geq  0}\cup\{ \infty\})^n$
are are compact and Hausdorff.
\end{remark}
The definitions 
of isomorphism, or automorphism, of extended tropical curves are given exactly in the same way as 
for tropical curves.
In particular,  
the
  automorphism group of an extended  tropical curve is finite.
For a stable graph $(\Gamma,w)$ we
denote 
$$R_{\infty}(\Gamma,w):=(\R_{> 0}\cup\{ \infty\})^{E(\Gamma)};$$
every  element in
$ R_{\infty}(\Gamma,w)$ corresponds to an extended tropical curve
whose underlying graph is  $(\Gamma, w)$.
We extend    the notation for tropical curves as follows.
We write $\Mt_{\infty}(\Gamma,w)$ for the set of isomorphism classes of extended tropical curves having $(\Gamma, w)$
as underlying graph; hence we have a surjection
$$
R_{\infty}(\Gamma,w)\la \Mt_{\infty}(\Gamma,w).
$$
The closure of $R_{\infty}(\Gamma,w)$ inside $(\R\cup\{ \infty\})^{E(\Gamma)}$ is denoted by
$$
\ov{R_{\infty}(\Gamma,w)}:=(\R_{\geq  0}\cup\{ \infty\})^{E(\Gamma)}.
$$
Lemma~\ref{cl} trivially extends,
so that we have a decomposition  
$$
\ov{R_{\infty}(\Gamma,w)}=\bigsqcup_{I\subset \{1,2,\ldots,E(\Gamma)\}}\ov{R_{\infty}(\Gamma,w)}_I
$$
where for every $I$ we have 
$ 
\ov{R_{\infty}(\Gamma,w)}_I\cong R_{\infty}(\Gamma_{/S_I},w_{/S_I}).
$ 

 We continue   
 by introducing the quotient space
$$
\ov{\pi_{\infty}}:\ov{R_{\infty}(\Gamma,w)}\la \ov{\Mt_{\infty}(\Gamma,w)}:=\ov{R_{\infty}(\Gamma,w)}/\sim 
$$
 where 
$ 
p_1\sim p_2$ if and only if
the 
 extended tropical curves parametrized by $p_1$ and $p_2$ are isomorphic.
 Naturally, $\ov{\Mt_{\infty}(\Gamma,w)}$ is endowed with the quotient topology
 induced by $\ov{R_{\infty}(\Gamma,w)}$.
The finite group $\Aut(\Gamma,w)$
acts   on $\ov{R_{\infty}(\Gamma,w)}$ by permuting the coordinates, and hence
it acts as group of homeomorphisms.
 Proposition~\ref{ovM}
extends:
\begin{prop}
\label{ovMg}
Let $(\Gamma,w)$  be a stable graph.
\begin{enumerate}
\item There is a decomposition
$$
\ov{\Mt_{\infty}(\Gamma,w)}=\bigsqcup_{(\Gamma',w')\leq (\Gamma,w)} \Mt_{\infty}(\Gamma',w'),
$$
where $\Mt_{\infty}(\Gamma,w)$ is   open and dense  in $\ov{\Mt_{\infty}(\Gamma,w)}$.
\item The  map $\ov{\pi_{\infty}}$ has finite fibers.
\item
 $\ov{\Mt_{\infty}(\Gamma,w)}$ is  a compact Hausdorff topological space.
 \end{enumerate} 
\end{prop}
\begin{proof}
The only new statement with respect to  Proposition~\ref{ovM}
is the compactness of $\ov{M_{\infty}(\Gamma,w)}$,
which follows from its being a quotient of the compact topological space $\ov{R_{\infty}(\Gamma,w)}$.

For the rest,  it suffices  to explain what  needs to be modified in the proof of Proposition~\ref{ovM},
whose notation we continue to use here.
One easily checks that
the only part where some changes are needed is the proof of Hausdorffness.
Indeed, we need to replace the open neighborhoods $B_{p}(\epsilon)$
with neighborhoods of a different type.
To do that, pick a point 
$p=(t_1,\ldots, t_{|E(\Gamma)|})\in \ov{R_{\infty}(\Gamma,w)}$, and any 
pair of positive real numbers $\epsilon, \eta$; 
we define the following open neighborhood of $p$
$$
A_p(\epsilon, \eta)=\prod_{j=1}^{|E(\Gamma)|}A_p(\epsilon, \eta)_j\subset (\R_{\geq 0}\cup \{ \infty\})^{E(\Gamma)}
$$
where $A_p(\epsilon, \eta)_j\subset \R_{\geq 0}\cup \{ \infty\}$ is the following open subset
$$A_p(\epsilon, \eta)_j=\begin{cases}
(t_j-\epsilon, t_j+\epsilon)\cap \R_{\geq 0}  & \text{ if }  t_j\in \R    \\
\R_{\geq 0}\cup \{ \infty\}\smallsetminus [-\eta,  +\eta] &  \text{ if }  t_j=\infty.
\end{cases}
$$
It is clear that $A_p(\epsilon, \eta)$ is open for all $p$, $\epsilon$ and $ \eta$ as above.

Now we show that $\ov{\Mt_{\infty}(\Gamma,w)}$ is Hausdorff,
 by  small modifications of the proof of Proposition~\ref{ovM}. 
Let $\ov{p}$ and $\ov{q}$ be two distinct points in  $\ov{M_{\infty}(\Gamma,w)}$. Let $\pi_{\infty}^{-1}(p)=\{p_1,\dots, p_n\}$
and $\pi_{\infty}^{-1}(q)=\{q_1,\dots, q_m\}$. It is easy to check that there
  exist  $\epsilon  \in \R_{>0}$ small enough and $\eta \in \R$ big enough so that  
 
\noindent
 (1) $A_{p_i}(\epsilon, \eta)\cap  A_{q_j}(\epsilon, \eta)=\emptyset$
  for every $i=1,\ldots, n$ and $j=1,\ldots, m$.

\noindent
(2) If $\ov{F_{\infty, I}}\cap A_{p_i}(\epsilon, \eta)\neq \emptyset$ then $p_i\in \ov{F_{\infty, I}}$. 
If 
 $\ov{F_{\infty, I}}\cap A_{q_j}(\epsilon, \eta)\neq \emptyset$ then $q_j\in \ov{F_{\infty, I}}$,
 where $ F_{\infty, I}:= {R_{\infty}(\Gamma,w)}_I$.

Set $U :=\cup_{i=1}^n A_{p_i}(\epsilon, \eta)$ and
$V=\cup_{j=1}^mA_{q_j}(\epsilon, \eta)$. 
It is clear that $U$ and $V$ are open and disjoint. We claim that 
\begin{equation}
\label{satg}
\ov{\pi}_{\infty}^{-1} (\ov{\pi}_{\infty}(U))=U \quad \text{and} \quad \ov{\pi}_{\infty}^{-1} (\ov{\pi}_{\infty}(V))=V.
 \end{equation}
 The proof is identical to the proof of the analogous claim  used for Proposition~\ref{ovM}.
 And (\ref{satg}) implies the statement, as in Proposition~\ref{ovM}.
  \end{proof}

Similarly to what we did in (\ref{Mgt}) we set
\begin{equation}
\label{Mgtb}
\Mgtnb:= \Bigl(\  \bigsqcup_{\stackrel{\Gamma \  3-{\text{regular}}}{b_1(\Gamma)=g}}\ov{\Mt_{\infty}(\Gamma,\underline{0})}\Bigr)/\cong
\end{equation}
where $\cong$ is isomorphism of extended tropical curves.
We obviously have a quotient map extending the map (\ref{qg})
$$
\ov{\pi_g}: \bigsqcup_{\stackrel{\Gamma \  3-{\text{regular}}}{b_1(\Gamma)=g}}\ov{\Mt_{\infty}(\Gamma,\underline{0})} \la \Mgtnb\supset \Mgtn 
$$
so that $\Mgtnb$ has the quotient topology induced by $\ov{\pi_g}$.
If $\Gamma$ is a  3-regular graph with $b_1(\Gamma)=g$,
then $\pi_g^{-1}(\pi_g([(\Gamma,\underline{0}, l)])) = [(\Gamma,\underline{0}, l])$.
 Extending Theorem~\ref{Mgt} we have
\begin{thm}
  \label{Mgtbthm} 
  Let $g\geq 2$.
 The points of   $\Mgtnb$ bijectively parametrize isomorphism classes of extended $n$-pointed tropical curves
 of genus $g$.
 The   topological space $\Mgtnb$ is   compact and Hausdorff. 
 Furthermore, it has pure dimension  $3g-3+n$ and it   is connected through codimension one\footnote{See  Definition~\ref{connconn}}.
\end{thm}
  \begin{proof}
  By Proposition~\ref{ovMg} the spaces $ \ov{\Mt_{\infty}(\Gamma,\underline{0})} $
  are all compact. Therefore 
 $\Mgtnb$ is  the quotient of a compact space
(the union in (\ref{Mgtb}) is   finite) and hence it is 
 compact.
The proof of its being connected through codimension one is the same as for $\Mgtn$, given in \ref{connt1};
indeed that proof is concerned with   the underlying weighted graphs, and completely ignores the length function.
The proof of its being Hausdorff 
 is a simple modification of the proof of  Theorem~\ref{Mgt}, along the same lines as the modifications used to prove Proposition~\ref{ovMg}
\end{proof}

\section{Comparing  moduli spaces.}
\label{compsec}
\subsection{Dual graph of a curve and  combinatorial partition of $\Mgnb$.}
We recall the definition of a useful graph associated to a pointed  nodal curve. \begin{defi}
The {\it (weighted) dual graph} of a pointed  nodal curve
$(X,\underline{p})=(X;p_1,\ldots,p_n)$
 is  the weighted graph $(\Gamma_{(X,\underline{p})},w_{(X,\underline{p})})$ with  $n$ legs
such that $\Gamma_{(X,\underline{p})}$ has a vertex for every irreducible component of $X$ and an edge for every node of $X$ joining the two (possibly equal) vertices corresponding to the components on which the node lies;
for every point $p_i$ there is a leg of $\Gamma_{(X,\underline{p})}$ adjacent to the vertex corresponding to the component containing $p_i$.
The weight function $w_{(X,\underline{p})}$ assigns to every vertex the geometric genus of the corresponding component.
\end{defi}
So the (arithmetic) genus $g$ of $X$
 is the same as that of  its dual graph
$$
g=
b_1(\Gamma_{(X,\underline{p})})+\sum_{v\in V(\Gamma_{(X,\underline{p})})}w_{(X,\underline{p})}(v)=g(\Gamma_{(X,\underline{p})}, w_{(X,\underline{p})}).
$$
\begin{remark}
\label{stablec}
The $n$-pointed curve $(X;\underline{p})$ is stable  if and only if its  weighted dual graph
$(\Gamma_{(X,\underline{p})},w_{(X,\underline{p})})$   is  stable. 
\end{remark}

\begin{example}
Assume $g=0$ and $n=4$. We described all stable $4$-pointed rational  curves in Example~\ref{04}.
Their dual graphs are pictured in Example~\ref{04g}.
\end{example}

Now, for any stable  graph $(\Gamma,w)$ of genus $g$ with $n$ labeled legs, we denote by $M^{\rm{alg}}(\Gamma, w)\subset \Mgnb$ the   locus of $n$-pointed curves with $(\Gamma, w)$ as dual graph:
$$
M^{\rm{alg}}(\Gamma, w):=\{(X;\underline{p})\in \Mgnb: (\Gamma_{(X,\underline{p})},w_{(X,\underline{p})})= (\Gamma ,w)\}\footnote{The superscript ``${\rm{alg}}$''  stands for ``algebraic'' and is used to keep the distinction with $\Mt(\Gamma, w)$.}.
$$ 
The following is    the {\it combinatorial partition} of $\Mgnb$:
\begin{equation}
\label{stratMg}
\Mgnb=\bigsqcup_{(\Gamma,w) {\text{ stable, $n$ legs,    genus }} g}M^{\rm{alg}}(\Gamma, w).
\end{equation}
The next well known fact is an easy consequence of Fact~\ref{mgnb}.

\begin{lemma}\label{dimstrata}
Assume  $2g-2+n>0$.
Then for any stable graph of genus $g$ with $n$ legs, $(\Gamma,w)$, we have that
$M^{\rm{alg}}(\Gamma,w)$ is an irreducible quasiprojective variety and its  codimension in
$\Mgnb$ is equal to $|E(\Gamma)|$.
\end{lemma}
\begin{proof}
Set  $\delta:=|E(\Gamma)|>0$ (otherwise we have a special case of \ref{mgnb}).
Pick a curve $(X;\underline{p})\in M^{\rm{alg}}(\Gamma,w)$, denote by $C_1,\ldots, C_{\gamma}$ its irreducible components, and let $n_i$ be the number of marked points contained in $C_i$,
so that $\sum_1^{\gamma}n_i=n$.

Our $(X;\underline{p})$
determines,
for  every $i=1,\ldots, \gamma$, a  nonsingular curve
of genus $g_i$
  with $n_i+\delta_i$ marked points in it,  $(C_i^{\nu};\underline{p}^{(i)})$,
where $\nu:\sqcup_1^{\gamma} C_i^{\nu}\to X$ denotes the normalization of $X$ and 
$$\delta_i:=|\nu^{-1}(\sing)\cap C_i^{\nu}|.$$
Moreover, as we observed in Remark~\ref{ind},
we have that $(C_i^{\nu};\underline{p}^{(i)})$ 
 is a stable curve, i.e. $(C_i^{\nu};\underline{p}^{(i)})\in M_{g_i, n_i+\delta_i}$.
Once we  have such $\gamma$ pointed curves, the gluing data of
$\nu^{-1}(\sing)$ are uniquely determined by the graph $\Gamma$.
We conclude that there is a surjective morphism
$$
M_{g_1,n_1+ \delta_1}\times\ldots\times M_{g_{\gamma},n_{\gamma}+\delta_{\gamma}}\la M^{\rm{alg}}(\Gamma,w).
$$

The above morphism extends to  $\ov{M}_{g_1,n_1+ \delta_1}\times\ldots\times \ov{M}_{g_{\gamma},n_{\gamma}+\delta_{\gamma}}\to \Mgnb$ 
and $ M^{\rm{alg}}(\Gamma,w) 
$ is open in its image; hence $ M^{\rm{alg}}(\Gamma,w) 
$ is quasiprojective. Now
  for every $i$,    Fact~\ref{mgnb} applies,
 hence $M_{g_i, n_i+\delta_i}$ is irreducible of dimension
$3g_i-3+n_i+\delta _i$ for all $i=1,\ldots, \gamma$.
Hence $ M^{\rm{alg}}(\Gamma,w)$ is irreducible.
Since the above surjection has clearly finite fibers  we get
$$
\dim M^{\rm{alg}}(\Gamma,w)=\sum _{i=1}^\gamma(3g_i-3+n_i+\delta _i)=3\sum _{i=1}^\gamma g_i-3\gamma
+ n+2\delta
$$ 
(since $\sum _{i=1}^\gamma\delta_i=2\delta$). Now $g=\sum _{i=1}^\gamma g_i+\delta -\gamma+1$ hence
$$
\dim M^{\rm{alg}}(\Gamma,w)=3g -3\delta +3\gamma -3
-3\gamma+n+2\delta=3g-3+n-\delta.
$$
\end{proof}

\subsection{Partition analogies}
By Theorem~\ref{Mgt} and Proposition~\ref{ovM} we have   on $\Mgtn$
a partition similar to (\ref{stratMg})
\begin{equation}
\label{stratMgt}
\Mgnt=\bigsqcup_{(\Gamma,w) {\text{ stable, $n$ legs,    genus }} g}\Mt(\Gamma, w).
\end{equation}  
In the next statement we highlight some topological and combinatorial analogies which, undoubtedly, are known to the experts. We include a (partial) proof, as we believe it is quite   instructive.
By $\dim M^{\rm{alg}}(\Gamma, w)$ we mean the dimension as an algebraic variety,
while
by $\codim \Mt(\Gamma, w)$ we mean the orbifold  codimension in $\Mgtn$.

\begin{thm}
\label{corr}
Assume $2g-2+n\geq 1$. Consider the partitions (\ref{stratMg}) and (\ref{stratMgt}), and 
the bijection  
$$
M^{\rm{alg}}(\Gamma, w)\mapsto \Mt(\Gamma, w)
$$
where $(\Gamma,w)$ varies among   stable graphs of genus $g$  with $n$ legs.
Then the following properties hold.
\begin{enumerate}
\item
\label{corr1}
$\dim M^{\rm{alg}}(\Gamma, w)=\codim \Mt(\Gamma, w)=3g-3+n-|E(\Gamma)|$.
\item
\label{corr2} With the notation   (\ref{min}),
$$
M^{\rm{alg}}(\Gamma, w)\subset\ov{M^{\rm{alg}}(\Gamma', w')}
\Leftrightarrow \Mt (\Gamma', w')\subset\ov{\Mt(\Gamma, w)}
\Leftrightarrow (\Gamma, w)\geq (\Gamma', w').$$
\item
\label{corr3}
Let $(\Gamma,w)$ be such that  $E(\Gamma)\neq \emptyset$. 
Then there exists a stable graph $(\Gamma',w')$ with $|E(\Gamma')|= |E(\Gamma)|-1$ such that 
$M^{\rm{alg}}(\Gamma, w)\subset\ov{M^{\rm{alg}}(\Gamma', w')}$
and  $\Mt (\Gamma', w')\subset\ov{\Mt(\Gamma, w)}$.
\end{enumerate}
\end{thm}
  \begin{proof}
 For (\ref{corr1}), the statement about $M^{\rm{alg}}(\Gamma, w)$ follows from Lemma~\ref{dimstrata}.
On the other hand, we observed above that  $
\dim \Mt(\Gamma,w)=|E(\Gamma)|$; hence we get    $\codim \Mt(\Gamma,w)=\dim \Mgtn-|E(\Gamma)|=3g-3+n-|E(\Gamma)|$.

Now   (\ref{corr2}), whose second double implication  follows from Proposition~\ref{ovM}(\ref{ovMc}).
We give the proof of the remaining part only in case $n=0$, leaving the (straightforward) generalization to the reader.

Let  $X\in M^{\rm{alg}}(\Gamma, w)$ be a stable curve; if $X$ lies in the closure of  
$M^{\rm{alg}}(\Gamma', w')$ there exist families of curves  with dual graph  $(\Gamma', w')$ specializing to $X$. We pick one of these  families,
$f:{\mathcal X}\to B$ with a one dimensional base $B$ with a  marked point $b_0$,
so that the fiber of $f$ over every  $b\neq b_0$ is a
 stable curve $X_b\in M^{\rm{alg}}(\Gamma, w)$, while the fiber over $b_0$ is isomorphic to $X$.

Under such a specialization every node of $X_b$ specializes to a node of $X$,
and  distinct nodes specialize to distinct nodes.
This singles out a  set $T \subset E(\Gamma)$
of nodes of $X$,
 namely 
$T$ is the set of nodes  that are
specializations of nodes of $X_b$. Let $S=E(\Gamma)\smallsetminus T$, so $S$ parametrizes the nodes of $X$ that do not come from nodes of $X_b$.
Consider the graph $(\GS,\wS)$;
we claim that $(\GS,\wS)=(\Gamma',w')$.
To prove it we shall use the notation of Definition~ \ref{cont}.
 
By construction we have a bijection between $E(\Gamma')$ and $E(\GS)=T$,
 mapping an edge of $\Gamma'$, i.e. a node of $X'$, to the node in $X$
 to which it specializes.
 
 The total space ${\mathcal X}$ of our family of curves   is singular along the nodes of the fibers $X_b$,
for $b\neq b_0$. Let us desingularize ${\mathcal X}$ at such loci (exactly $|T|$ of them);
 we thus obtain a new family
 ${\mathcal Y}\to B$ whose fiber over $b\neq b_0$ is the normalization of $X_b$.
 The   fiber over $b_0$ is the partial normalization  of $X$ at $T$, which we denote by $Y$; so its dual graph  satisfies
 $$
 \Gamma_{Y}=\Gamma - T.
 $$
Notice that  ${\mathcal Y}\to B$ is a
 union of families parametrized by the irreducible components of $X_b$, i.e. by the vertices of $\Gamma'$.
 Let us denote these families by ${\mathcal Y}({v'})\to B$. So if $b\neq b_0$  the fiber of ${\mathcal Y}({v'})$ over $b$ is the smooth irreducible component corresponding to  ${v'}\in V(\Gamma')$. 
 The fiber over $b_0$ of ${\mathcal Y}(v')\to B$ 
 is a 
 connected component of $Y$, which we denote $Y(v')$.
Of course $Y(v')$
 determines  a set of vertices of $\Gamma$
 (those corresponding to its components). 
Now notice that  two different vertices  of $\Gamma'$ determine in this way disjoint sets of vertices of $\Gamma$.
 Therefore we have a surjection $\phi:V(\Gamma)\to V(\Gamma')$
 mapping each vertex $v$ to the vertex $v'$ such that the component corresponding to $v$ lies  in $Y(v')$.  
 It is clear that $\phi(v_1)=\phi(v_2)$
  if and only if $v_1$ and $v_2$ belong to the same connected component of $\Gamma\smallsetminus T$.
Therefore $\phi$ is the same map as the map $\sigma_V:V(\Gamma)\to V(\GS)$.
This shows that $V(\Gamma')$ and $V(\GS)$ are in natural bijection, and hence that
$\Gamma'\cong \GS$.
Finally, since the arithmetic genus of a family of algebraic curves is constant, we have, for any $v'\in V(\Gamma')$,
that the genus of the component  corresponding to $v'$, i.e. the weight $w'(v')$, 
is equal to the arithmetic genus of the limit curve $Y(v')$. Therefore
$$
w'(v')=b_1(\Gamma_{Y(v')})+\sum_{v\in \sigma_V^{-1}(v')}w(v)=b_1({\sigma^{-1}(v')})+\sum_{v\in \sigma_V^{-1}(v')}w(v).
$$
By (\ref{wS}) the weight function $w'$ coincides with $\wS$; so we are done.

Conversely, suppose that $(\Gamma, w)\geq (\Gamma', w')$ with  $ (\Gamma', w')= (\GS, \wS)$
for some $S\subset E(\Gamma)$; let $T:=E(\Gamma)\smallsetminus S$.
We shall show how to reverse the procedure we just described.
Let $X\in M^{\rm{alg}}(\Gamma, w)$ and let $Y\to X$ be the normalization of $X$ at $T$, so that 
$\Gamma_Y=\Gamma-T$.
Notice that $Y$ is endowed with $|T|$ pairs of smooth distinguished points, namely the branches over the nodes in $T$, and it is thus a disjoint union of stable
pointed curves. Therefore, by Remark~\ref{ind}, there exists a family of curves $\Y\to B$ with $|T|$ pairs of disjoint sections
(with $\dim B=1$,  $b_0\in B$, and $\Y$ a possibly disconnected surface)
whose fiber over $b_0$ is $Y$ (as a pointed curve) and
  whose fiber over $b\neq b_0$ is a disjoint union of smooth curves with $2|T|$ distinct points.
  We let $\X$ be the surface obtained by gluing together the $|T|$ pairs of sections.
  It is clear that $\X$ is a family of nodal curves over $B$, whose fiber over $b_0$ is $X$ and whose fiber over $b\neq b_0$ lies in $M^{\rm{alg}}(\Gamma',w')$.
  
  Now (\ref{corr3}). Let $e\in E(\Gamma)$ and set $(\Gamma',w')=(\Gamma_{/e},w_{/e})$. Then $(\Gamma',w')$ is stable and
  $|E(\Gamma')|=|E(\Gamma)|-1$. By (\ref{corr2}) we are done.
\end{proof}
 
 \begin{cordef}
 \label{kstrat}
The $k$-{\emph {dimensional stratum}} of the combinatorial partition of $\Mgnb$ given in (\ref{stratMg})
is defined as the following closed subscheme of $\Mgnb$:
\begin{equation}
\Cgk:=\ov{\bigsqcup_{\dim M^{\rm{alg}}(\Gamma,w)= k}M^{\rm{alg}}(\Gamma, w)}=\bigsqcup_{\dim M^{\rm{alg}}(\Gamma,w)\leq k}M^{\rm{alg}}(\Gamma, w).
\end{equation}
\end{cordef}
\subsection
{Connectedness properties.}
\label{scho}
The next definition is adapted from  \cite[Definition 3.3.2]{MS}.
\begin{defi}
\label{connconn}
Let $X$ be a topological space of  pure dimension $d$;
see Remark~\ref{dim}.  
Assume that $X$ is endowed with a decomposition $X=\sqcup_{i\in I}X_i$,
where every $X_i$ is a connected orbifold.
 We say that $X$ is {\it connected through codimension one} if the subset
 $$
 \bigsqcup_{i\in I: \dim X_i\geq d-1}X_i\subset X
 $$
 is connected.
\end{defi}
Connectedness through codimension one is a strong form of connectedness (if $X$ is   connected through codimension one, it is also
 connected). 
 Tropical varieties  (associated to prime ideals)  have this property, which in fact is a fundamental one;
 see the  Structure Theorem  in \cite[Ch. 3]{MS}.

Although $\Mgtn$ is not  a tropical variety in general,  we may still  ask whether $\Mgnt$ is connected through codimension one.
 The answer   is yes (see \cite{BMV} for the unpointed case).
 
 The second part of the next proposition, stating that $\Cgk$ is connected,  may be known, although we don't have a reference;   our  proof,  exhibiting as a a simple consequence of the first part,  is a good    illustration of
 the interplay between the combinatorial and the algebro-geometric point of view.
 \begin{prop}
\label{connt1}
\begin{enumerate}
\item
$\Mgnt$ is connected  through codimension one.
\item
If $k\geq 1$ the $k$-dimensional stratum $\Cgk$ of the combinatorial partition of $\Mgnb$ is connected.
\end{enumerate}\end{prop}
\begin{proof}
By Theorem~\ref{corr}, the second part in case $k=1$ is an immediate consequence of the first,
and the case $k\geq 2$ follows from it.

So, let us concentrate on the tropical moduli space.
We saw  in
\ref{puremt}
that $\Mgnt$ is of pure dimension $3g-3+n$.  
We know that (\ref{stratMgt}) is a decomposition of $\Mgtn$  as disjoint union of orbifolds 
of known dimensions.
We must therefore show that its subset
$$
 \bigsqcup_{\codim\Mt(\Gamma,w)\leq 1} \Mt(\Gamma, w)= \bigsqcup_{|E(\Gamma)|\geq 3g-4+n} \Mt(\Gamma, w)
$$   is connected.
Recall that in the above decomposition the top dimensional orbifolds, of dimension $3g-3+n$, are
precisely the ones for which  $\Gamma$ is a $3$-regular graph.
To prove the statement we apply   the following result of Hatcher-Thurston  \cite[Prop. page 236]{HT}
(a purely combinatorial proof, with other applications to moduli of tropical curves, may be found in   \cite{Ctrop}): 
\begin{fact}
\label{linkage}
Let  $\Gamma$ and $\Gamma '$ be two $3$-regular connected  graphs
   of the same genus, free from legs. Then 
there exists a finite sequence
\begin{equation}
\xymatrix{
\Gamma=\Gamma_1 \ar[rd] ^{ }&&\Gamma_3\ar[ld] _{}\ar[rd]^{}&\ldots &  \ldots&\ar[ld]_{}\Gamma_{2h+1}=\Gamma ' \\
 & \Gamma_2   
 &  &\ldots& \Gamma_{2h}&\\
}
\end{equation}
where every arrow is the  map  contracting precisely one  edge which is not a loop. Moreover, every odd-indexed graph
in the diagram above is $3$-regular.
\end{fact}
More precisely,  every even-indexed graph above satisfies
$$
\Gamma_{2i}=(\Gamma_{2i-1})_{/e}=(\Gamma_{2i+1})_{/e'}
$$
with $e\in E(\Gamma_{2i-1})$ and $e'\in E(\Gamma_{2i+1}).$

To use this for our statement, let us first suppose that $n=0$.
Then we may  apply the above diagram to   two weighted graphs  $(\Gamma,\mo)$ and $(\Gamma',\mo)$.
Since   no loop gets contracted,
we may view the arrows of the above diagram as weighted contractions,
where the weight function is always the zero function.
By hypothesis 
the odd-indexed sets $\Mt(\Gamma_{2i-1},\mo)$ are  codimension zero sets 
of the decomposition (\ref{stratMgt}), hence
every $\Mt(\Gamma_{2i},\mo)$  has codimension one.

The above fact says    that 
 the closures of two consecutive odd-indexed sets
$\Mt(\Gamma_{2i-1},\mo)$ and $\Mt(\Gamma_{2i+1},\mo)$ intersect in  $\Mt(\Gamma_{2i},\mo)$.
Therefore, if $n=0$,  we are done.

If $n>0$ we may assume $3g-3+n\geq 2$, for otherwise the result follows from what we know already
(namely, that $\Mgnt$ is connected).
 Now, we need to show that   Fact~\ref{linkage} holds for graphs with $n$ legs, in such a way that no leg gets contracted by the maps in the diagram.
 This is  easily proved by induction on $n$, noticing that
if we remove from  our $3$-regular graph  $\Gamma$
a leg together with its base vertex, we obtain a $3$-regular graph of the same  genus as $\Gamma$,
and $n-1$ legs
(as $3g-3+n\geq 2$).  Conversely, if we add a leg in the interior of any edge of a $3$-regular graph with $n-1$ legs, we get a $3$-regular graph with $n$ legs, and all $3$-regular graphs with $n$ legs can be obtained in this way (more details can be found in \cite[Prop. 3.3.1]{Ctrop}).
\end{proof}

\section{Moduli spaces via Teichm\"uller theory}
\label{secteich}
From now on, we restrict to the unpointed case $n=0$.
Up to now we considered the moduli space of stable algebraic curves $\Mgb$,
or, which is the same, the moduli space of stable equivalence classes of nodal curves.
The  construction of $\Mgb$ as a projective variety, as given in \cite{Gie},
can be summarized as follows
(we refer also to \cite[Chapt. 4]{HM} for details and references). 

By definition,
stable curves have an ample dualizing bundle,
so they can all be embedded in  some projective space using a suitable power of it.
Furthermore, this projective space can be chosen to be the same 
$\pr{r}$ for all stable curves of fixed genus 
$g$. 
Since the dualizing bundle is preserved by isomorphisms, two such projective curves are abstractly isomorphic if and only if they are
projectively equivalent, i.e. conjugate by an element of the group $G=\Aut(\pr{r})=PGL(r+1)$.

Let us denote by $H$ 
 the set of all these projective models of our curves.
By what we said 
there is an obvious bijection between 
the quotient set $H/G$  and the set of isomorphism classes of genus $g$ stable curves.
 
To give the set $H/G$ an algebraic structure,
one proceeds by   giving $H$ (via Grothendieck's theory of Hilbert schemes), and then $H/G$  
(via Mumford's Geometric Invariant Theory) the structure of an algebraic variety.

The approach we just sketched has many advantages;
we wish to mention a few of them.
One  is the fact that stable curves are treated at the same time as  smooth curves,
so that  the same construction yields   a projective algebraic variety $\Mgb=H/G$, the moduli space of stable curves, and an open dense subset $M_g\subset \Mgb$,
  the moduli space of smooth curves.

Another consequence of this construction is that, being purely algebraic, it works in any characteristic. Since the space $H$ turns out to be irreducible and smooth,
one obtains that $\Mgb$, and hence $M_g$, is irreducible (an important fact that was not known in positive characteristic) and normal, i.e. mildly singular.

Finally, the quotient $H\to H/G$ can be rather explicitly described, locally at every point.
By what we said, the stabilizers of the action of $G$ are the automorphism groups of our abstract stable curves, which are finite by definition.
In particular, locally at curves having no nontrivial automorphisms (and there is a dense open subset of them as soon as $g\geq 3$) the space $\Mgb$ is nonsingular.

A somewhat more sophisticated  construction of $\Mgb$    can be given, by first
constructing the related algebraic  stack, 
$\Mgbst$, and then showing that this stack admits a projective  moduli scheme $\Mgb$. The stack $\Mgbst$ is preferable from the point of view of moduli theory, as it retains
more information about the moduli problem than the moduli variety $\Mgb$. 
Stacks form a larger category than algebraic varieties or schemes, and may be viewed, loosely speaking, as the algebraic counterparts of topological orbifolds.
We have mentioned this more general point of view as it relates to some of the open problems that we shall list at the end of the paper. More details can be found in \cite{DM} or \cite[Chapt 12 and 14]{gac}.

We will end the section by speaking about another approach  to  the  moduli theory of curves
that can be used both in the  algebraic and tropical situation.

\subsection{The Teichm\"uller approach to   moduli of complex curves}
In complex geometry,
a different, quite natural,  point of view   is the so-called
Teichm\"uller approach, which puts emphasis on the topological aspects. 
We shall now give a very short overview of   profound and fundamental results; for   details    and references we refer to \cite[Chapt. XV]{gac}. 

The starting observation is that every smooth complex curve $C$ of genus $g$ has
the same underlying topological manifold, up to homeomorphism;
namely a  compact, connected, orientable surface of genus $g$, which we shall fix and denote by $S_g$
from now on. 

These topological $2$-manifolds have been deeply  studied in the past, and their classification dates back to the nineteenth century. 
Moreover, the group of isotopy classes of orientation-preserving homeomorphisms\footnote{Two homeomorphisms are isotopic if they can be deformed to one another by a continuous family of homeomorphisms.} of $S_g$ with itself, known as 
the   {\it mapping class group} and denoted by
$\Gamma_g$, has also been 
heavily studied. Among its properties, we need to recall that, denoting by $\Pi_g:=\pi_1(S_g)$, we have
$$
\Gamma_g\cong \Out^+(\Pi_g)\subset \Out(\Pi_g):=\Aut(\Pi_g)/\Inn(\Pi_g) 
$$
where the superscript ``$+$" indicates restriction to orientation preserving automorphisms (see loc. cit. for the precise definition).
Let us go back to our moduli problem. To parametrize complex curves one proceeds by
``marking them", i.e. one fixes a homemorphism, called  {\it marking}, $\mu:S_g\to C$, up to isotopy, and considers the set of isomorphism classes of marked complex curves $(C,\mu)$.
So, $(C,\mu)$ and $(C',\mu')$ are in the same class if there exists an algebraic isomorphism $\iota:C\to C'$
such that the diagram below commutes
\begin{equation}
\xymatrix{
C\ar[rr]_<(0.5){\cong}^<(0.5){\iota}& & C' \\
&S_g\ar[ru] _>(0.6){\mu'}  \ar[lu] ^>(0.6){\mu} 
}
\end{equation}
The {\it Teichm\"uller space} $T_g$ is defined as the set of all such classes
$$
T_g:=\{[(C,\mu)], \  C \  \text{ of genus } g \}.
$$
Using the basic deformation theory of curves,   in particular 
the properties of their so-called Kuranishi family, the Teichm\"uller space $T_g$ is endowed with a natural structure of complex manifold
of dimension $3g-3$. $T_g$ is  known to be a contractible space.

Now, it is clear that there is a surjection
$$
T_g\la M_g;\  \  \  [(C,\mu)]\mapsto [C].
$$
In fact, this surjection is the quotient of $T_g$ by the natural action 
of the mapping class group $\Gamma_g$ given by, for $\gamma\in \Gamma_g$
$$
[(C,\mu)]\mapsto [(C,\gamma\circ\mu)].
$$
Let us conclude this brief description by mentioning that the Teichm\"uller approach has had remarkable applications
in establishing some geometric properties   of $M_g$ and of $M_{g,n}$; see \cite{HM} or  \cite{gac}.

\subsection{The Teichm\"uller approach to   moduli of metric graphs}
We shall now briefly overview some work of  M. Culler and K. Vogtmann (\cite{CuVo}) on moduli of metric graphs.
Their approach, inspired  by the previously  described Teichm\"uller  theory, presents strong analogies with the algebraic situation. By contrast,
  their main goal was not the study of the moduli space of metric graphs, 
 which appears more as a tool, however important, than a goal.
Rather, the principal scope of \cite{CuVo} was to study an important  group:
 the automorphism group of the free group of rank $g$, $F_g$.

Following an extremely successful line of research,
 groups may be studied through their actions on some geometric spaces.
To study the automorphism group of $F_g$, or rather, its quotient by inner automorphisms,
$\Out(F_g)$, in \cite{CuVo} the authors introduce a space, called the {\it outer space} and denoted
here by $O_g$ (as in \cite{VICM})
on which $\Out(F_g)$ acts. 

The starting point is that $F_g$ is the fundamental group, $\pi_1(R_g)$, of the connected graph of genus $g$
with one vertex and $g$ loops, often called the ``rose with $g$ petals", and written $R_g$.

Elements in $\Out(F_g)$ are thus geometrically represented by homotopy equivalences of $R_g$ with itself.

Now, a connected graph $\Gamma$ of genus $g$ is  homotopy equivalent to $R_g$.
By adding a metric structure $\ell$ to $\Gamma$, and letting $\ell$ vary,   one obtains a new
 interesting space on which
$\Out(F_g)$ acts, provided  $\Gamma$ is ``marked", i.e. provided a homotopy equivalence
$\mu:R_g\to \Gamma$ is fixed.
Now, to ensure that the space of all marked genus $g$ metric graphs has finite dimension
one needs to assume that  graphs have no vertices of valence $1$ or $2$ (sounds familiar!).
An extra, normalization, assumption made in \cite{CuVo} is that the sum of the lengths of the edges,
i.e. the {\it volume} of the graph,
$ 
vol(\Gamma,\ell)=\sum_{e\in E(\Gamma)}\ell(e),
$   be equal to $1$.
The outer space $O_g$ is defined as follows
$$
O_g:=\{[(\Gamma,\ell,\mu)]\  \text{ of genus } g, \text{ volume }1, \text{ having no vertex of valence }\leq 2\} 
$$
 where $(\Gamma,\ell,\mu)$ and $(\Gamma',\ell',\mu')$ are in the same class if there exists an isometry between 
 $(\Gamma,\ell)$ and $(\Gamma',\ell')$ commuting with the markings up to homotopy..

So, $O_g$  plays the same role for  metric graphs   as the Teichm\"uller space $T_g$
for   complex curves. The group $\Out(F_g)$ plays the role of the mapping class group $\Gamma_g$.

In \cite{CuVo} the outer space $O_g$ is given the structure of a cell complex, 
it is shown not to be a manifold (not hard), and   to be a contractible space (a theorem, also attributed to Gersten).

The quotient space $O_g/\Out(F_g)$ is thus a moduli space for metric graphs of fixed volume and genus,   having no vertex of valence $\leq 2$. This quotient is thus  closely related to the moduli space of tropical curves.

This connection has not yet been thoroughly investigated,
perhaps because of   the relative  youth  of tropical geometry, or the complexity of the group $\Out(F_g)$. 
Also because, due to the recent, fast development of tropical geometry,
 interest in moduli of metric graphs is probably stronger nowadays
than it was twenty years ago.
We believe it 
 deserves to be   studied as  it offers an entirely      new   perspective.
For example, using what is known for the outer space $O_g$ (see the survey \cite{V})
 one may construct bordifications of the moduli space
 of pure tropical curves other
 than the one we described in the present paper (via weighted, or extended tropical curves).
 This may lead to new, profound  insights, just like it has been the case for
  the moduli space of smooth algebraic curves, for which different  compatifications exist and have been used in different ways.   

\subsection{Analogies}
The analogies between the   Teichm\"uller construction of $M_g$
and the Culler-Vogtmann space, are summarized in the following  table, where
the word ``homeomorphism" stands for ``orientation preserving homeomorphism".
\vfill
\begin{center}
\begin{tabular}{|c|c|}
\hline
&\\
ALGEBRAIC CURVES&METRIC GRAPHS\\
&\\
\hline
&\\
$S_g$ topological surface of genus $g$
&
$R_g$ connected  graph \\
compact connected orientable &with one vertex and $g$ loops.\\
&\\
\hline
&\\
$\Pi_g:=\pi_1(S_g)$
&
$F_g:=\pi_1(R_g) $\\
&\\
\hline
&\\
$C$ algebraic curve,  genus $g$ & $(\Gamma,\ell)$ metric graph,  genus $g$,  volume $1$\\
smooth  projective over $\C$
&connected, no vertex of valence $\leq 2$\\
&\\
\hline
&\\
a {\it marking} of $C$: &  a {\it marking} of $\Gamma$:\\
a homeomorphism up to isotopy &  a homotopy equivalence\\
 $\mu:S_g\la C$ & $\mu:R_g\la \Gamma$\\
&\\
\hline
&\\
The {\it Teichm\"uller Space}: & The {\it Outer Space}:  \\
$T_g:=\{(C,\mu)\}/\sim$  & $O_g:=\{(\Gamma,\ell,\mu)\}/\sim$ \\

$(C,\mu)\sim (C',\mu')$ if $\exists \  C\stackrel{\iota}{\to} C'$ & $(\Gamma,\ell,\mu)\sim (\Gamma ',\ell',\mu')$ if $\exists \  \Gamma\stackrel{\iota}{\to} \Gamma''$\\

$\iota$ an algebraic isomorphism    &  $\iota$ an   isometry\\
 
such that $\iota\mu=\mu'$    & such that $\iota\mu=\mu'$  up to homotopy\\
&\\
\hline
&\\
$T_g$ is a contractible $\C$-manifold & $O_g$ is a contractible cell complex \\
& not a manifold \\
$\dim T_g=3g-3$&$\dim O_g=3g-4$\\
&\\
\hline
&\\
The {\it mapping class group}: &   \\
$\Gamma_g=\{\phi:S_g\to S_g  $  homeomorphism$\}$&  $\{\phi:R_g\to R_g  $  homotopy equivalence$\}$\\
&\\
$\Gamma_g=\Out^+(\Pi_g)$&$\Out(F_g)$ \\
&\\
 \hline
&\\
 $\Gamma_g$ acts on $T_g$ with finite stabilizers& $\Out(F_g)$ acts on $O_g$ with finite stabilizers\\
&\\
 \hline
&\\
 $M_g=T_g/\Gamma_g $ moduli space& $O_g/\Out(F_g)$ moduli space  \\
of smooth genus $g$ algebraic curves & for genus $g$ metric graphs  of volume $1$ \\
&\\
 \hline
\end{tabular}
\end{center}

\section{Open problems.}
We conclude with some interesting lines of research closely related to 
the topics treated in this paper.
\begin{enumerate}[{\bf (1)}]
\item
\label{P1}
{\it Compactifications of $\Mgpnn$.}
We constructed the bordification of the moduli space of pure  pointed tropical curves using weighted   tropical curves (as done  in  \cite{BMV} for the unpointed case, using a different set up) 
and its compactification, by extended weighted  tropical curves.
\smallskip

\noindent{\bf Question 1.} 
\begin{enumerate}[(A)]
\item
What   geometric interpretations can these compactifications be given?
\item
Do there exist other compactifications, or bordifications? 
If so, how do they relate among each other?
\end{enumerate}
\smallskip

As we already mentioned, it should be especially interesting to study these issues in connection
with the set up described in Section~\ref{secteich}.

\smallskip
\item{\it Divisor theory.}
In analogy with the classical theory, for pure tropical curves 
there is a good notion of divisors, linear equivalence and linear systems.
In particular, the theorem of Riemann-Roch is known to hold for pure tropical curves,
by work of A. Gathmann and M. Kerber  \cite{GK}, building upon work of M. Baker and S. Norine (in \cite{BN}) for
combinatorial graphs; see also the work on G. Mikhalkin and I. Zharkov in \cite{MZ}.

One appealing research direction is to explore  the connection with the divisor theory for algebraic curves, smooth or singular. 
There is a clear relation between the combinatorial and the algebraic divisor theory. In fact, let $X$ be a nodal curve and $\G$ its dual graph; then for any divisor $D$ on $X$
its multidegree $\underline{\deg} D$ is naturally 
 a divisor on $\G$, and on any tropical curve supported on $\G$.
The interaction between the two theories is not trivial, and its investigation
 may bring new fertile perspectives also on the classical theory.
This in fact has already  happened,  as shown by the new proof of the famous Brill-Noether theorem
for algebraic curves 
given in  \cite{CDPR}, using the groundwork developed in  \cite{B}. 

Observe that    the ``Specialization Lemma" of \cite{B}, a powerful tool to relate the combinatorial with the algebro-geometric divisor theory,   can be 
 refined    to be applicable 
in the classical algebro-geometric set-up  (see \cite{CBNgraph}) and extended to weighted
graphs (see \cite{AC}).

\smallskip

A related problem   is the  extension of the Riemann-Roch theorem, and of other classical
theorems from Riemann surfaces,    to weighted tropical curves, 
More specifically:

\smallskip
\noindent
{\bf Question 2.} Does a Riemann-Roch theorem hold for line bundles on weighted tropical curves?
What about Clifford's theorem?
\smallskip

A natural approach to these  last problems  would  require, first of all, an answer 
to Question 1.A, raised in
part (\ref{P1}). A proof  of the Riemann-Roch formula for a definition of rank extending the one used in \cite{GK} is given in \cite{AC}.     Other definitions of rank are used, for which the Riemann-Roch formula is not known.

\

\item{\it Tropical orbifolds or stacks.}
Establish a categorical framework for tropical moduli theory.
This was also the leading theme of a school in tropical geometry,
organized by E. Brugall\'e, I. Itenberg and G. Mikhalkin, which took place in March 2010.
As we   emphasized a few times in the paper, the spaces $\Mgpnn$ and $\Mgtn$ are
not manifolds, and seem too complicated to be tropical varieties (with a few known exceptions).

\smallskip

\noindent{\bf Question 3.}  What is a good category to use in tropical geometry?
Can this category be defined so as to include tropical varieties?

\smallskip

Comparing with the algebraic setting, as we mentioned in Remark~\ref{moduli}, moduli spaces for algebraic curves can be studied within the category of  (singular) algebraic schemes, or, more generally, within that of Deligne-Mumford stacks
(but also, that of algebraic spaces, or that of orbifolds). 

A promising, seemingly unexplored,  line of research could  be, again, to study the connection
with the work in \cite{CuVo} and the later developments. 
\end{enumerate}
 
\end{document}